\documentclass[12pt,twoside]{article}
\usepackage[a4paper,width=160mm,top=30mm,bottom=25mm,bindingoffset=6mm]{geometry}

\usepackage{siunitx}
\usepackage{graphicx}
\usepackage{wrapfig}

\usepackage{comment}

\usepackage{bm}
\usepackage{amsmath}
\usepackage{amsthm}
\usepackage{amsfonts}

\newtheorem{theorem}{Theorem}
\newtheorem{claim}{Claim}
\newtheorem{proposition}[theorem]{Proposition}
\newtheorem{lemma}[theorem]{Lemma}
\newtheorem{corollary}[theorem]{Corollary}

\title{\bf{Scale-free spanning trees: complexity, bounds and algorithms}}

\date{}

\author{
        Yury Orlovich  \footnote{Faculty of Applied Mathematics and Computer Science, Belarusian State University, 220030, Minsk, Belarus}
        \and 
        Kirill Kukharenko \footnote{Institute for Mathematical Optimization, Otto von Guericke University Magdeburg, 39106, Magdeburg, Germany}
        \and
        Volker Kaibel \footnotemark[\value{footnote}]
        \and 
        Pavel Skums \footnote{Department of Computer Science, Georgia State University, Atlanta, GA 30303, USA}
}

\providecommand{\keywords}{\textbf{Keywords: }}

\usepackage{fancyhdr}
\renewcommand\abstractname{\uppercase{abstract}}

\renewenvironment{abstract}
 {\small
  \begin{center}
  \bfseries \abstractname\vspace{-.5em}\vspace{0pt}
  \end{center}
  \list{}{%
    \setlength{\leftmargin}{0mm}
    \setlength{\rightmargin}{\leftmargin}%
  }%
  \item\relax}
 {\endlist}
 
\renewenvironment{itemize}
{\begin{list}
             {\labelitemi}
             {\setlength{\labelwidth}{1em}
              \setlength{\labelsep}{0.7em}
              \setlength{\itemindent}{0em}
              \setlength{\listparindent}{3em}
              \setlength{\leftmargin}{3.6em}
              \setlength{\rightmargin}{0em}
              \setlength{\parsep}{0ex}
              \setlength{\topsep}{1ex}
              \setlength{\itemsep}{0ex}
             }%
}
{\end{list}}%

\begin{document}

\pagenumbering{arabic}
\maketitle

\begin{abstract}
\bf{We introduce and study the general problem of finding a most ``scale-free-like'' spanning tree of a connected graph. It is motivated by a particular problem in epidemiology, and may be useful in studies of various dynamical processes in networks. We employ two possible objective functions for this problem and introduce the corresponding algorithmic problems termed $m$-SF and $s$-SF Spanning Tree problems. We prove that those problems are APX- and NP-hard, respectively, even in the classes of cubic, bipartite and split graphs. We study the relations between scale-free spanning tree problems and the max-leaf spanning tree problem, which is the classical algorithmic problem closest to ours. For split graphs, we explicitly describe the structure of optimal spanning trees and graphs with extremal solutions. Finally, we propose two Integer Linear Programming formulations and two fast heuristics for the $s$-SF Spanning Tree problem, and experimentally assess their performance using simulated and real data.}
\end{abstract}

\keywords{scale-free network, spanning tree, optimal tree, combinatorial optimization, integer linear programming, NP-hardness.}

\setcounter{section}{0}
\section{Introduction and motivation}
\label{scale-free:sec1}

In the recent two decades, significant amount of research associated with applied graph-theoretical models has been dedicated to the so-called \emph{``scale-free''} graphs~\cite{barabasi1999emergence,bollobas2001degree,dorogovtsev2000structure}. The popularity of this concept originates from the fact that it seems to reflect important properties of graphs and networks arising in biology, social sciences, physics and engineering. It is usually assumed that a random scale-free graph possesses a particular set of properties, including a power-law degree distribution, a small diameter, presence of high-degree vertices and a certain self-similarity originated from the recursive probabilistic rule for its construction. 

The algorithmic and graph-theoretical problems studied in this paper originated from a problem from mathematical epidemiology~\cite{skums2017quentin}. Consider a graph $G$, whose vertices represent individuals infected by a virus, and edges represent the possibility of viral transmission between pairs of individuals (such possibilities are usually deduced by the experts from genetic or epidemiological evidence). The goal is to find the most probable transmission history (``who infected whom''). Under the assumption that each individual has been infected only once, feasible transmission histories correspond to spanning trees of $G$ (called \emph{transmission trees} in this context). It is known that for viruses, whose modes of transmission are associated with behavioral risk factors (e.g. HIV or Hepatitis C), their transmission trees have properties of scale-free graphs~\cite{wertheim2013global}. This observation gives rise to the following informally defined algorithmic problem (\emph{scale-free spanning tree problem}): given a graph $G$, find the most ``scale-free-like'' spanning tree of $G$. This problem may arise in other domains associated with the study of dynamical processes on scale-free networks (e.g. spread of information, opinion, etc.). 

In order to study the scale-free spanning tree problem, a mathematically rigorous definition of its objective function is required. Several non-equivalent definitions of scale-free graphs of various degree of mathematical rigour have been used in the literature. One of the most precise definitions allowing to incorporate or deduce most of the expected properties of scale-free graphs has been introduced in~\cite{li2005towards} using the so-called \emph{$s$-metric} of a graph. This graph invariant is defined as follows:
\begin{equation}
\label{eq:smetric}
s(G) = \sum\limits_{uv \in E(G)} \deg u \deg v.
\end{equation}
The same parameter is known in mathematical chemistry under the name \emph{second Zagreb index}~\cite{borovicanin2017bounds,das2004some}. A series of propositions proved in~\cite{li2005towards} demonstrates that in the space of random graphs with the same expected degree sequence, higher $s$-metric indicates with  high probability the presence of most of the expected properties of scale-free graphs. The intuition behind these results is that in graphs with high $s$-metric a large number of edges should be incident to high-degree vertices, thus forcing them to be structurally similar to graphs produced by preferential attachment process, which is a standard model of scale-free networks formation~\cite{barabasi1999emergence}. Given this observation, another classical mathematical chemistry parameter called the \emph{first Zagreb index}~\cite{borovicanin2017bounds} also can serve as a measure of ``scale-freeness'' of a graph. This parameter is defined as 
\begin{equation}
\label{eq:mmetric}
m(G) = \sum\limits_{u\in V(G)} (\deg u)^2\, = \sum\limits_{uv \in E(G)} (\deg u + \deg v).
\end{equation}

Thus, we can formulate two variants of the scale-free spanning tree problem:

\textsc{$m$-SF Spanning Tree}

\emph{Given}: A connected graph $G$.

\emph{Find}: A spanning tree $T$ of $G$ such that $m(T)$ is maximum.

\textsc{$s$-SF Spanning Tree}

\emph{Given}: A connected graph $G$.

\emph{Find}: A spanning tree $T$ of $G$ such that $s(T)$ is maximum.

Both problems are naturally associated with the \emph{first} and \emph{second SF-dimensions} of $G$ denoted by $\tau_1(G)$ and $\tau_2(G)$, respectively, and defined as follows:
\begin{equation}
\label{eq:sdim}
\tau_1(G) = \max\limits_{T\in \mathcal{T}(G)}\{m(T)\},\qquad \tau_2(G) = \max\limits_{T \in \mathcal{T}(G)}\{s(T)\},
\end{equation}
where the maximums are taken over the set $\mathcal{T}(G)$ of all spanning trees of $G$.

The related problem has been studied in~\cite{kincaid2016algorithms}. In that paper, the problem under consideration is, given a graph $G$, to find a spanning subgraph $H$ with \emph{prescribed vertex degrees} such that its $s$-metric is maximum. It has been demonstrated that this problem is polynomially solvable in general (by reduction to the $f$-factor problem~\cite{schrijver2003combinatorial}), but becomes NP-hard, when the additional constraint is added stating that the output spanning subgraph has to be connected.

In this paper, we present the first detailed study of the scale-free spanning tree problems from both theoretical and practical sides. Our contributions are summarized as follows. 

\begin{itemize}
    \item[1)] We establish the computational complexity of the \textsc{$m$-SF Spanning Tree} and \textsc{$s$-SF Spanning Tree} problems. We demonstrate that these problems are NP-hard or APX-hard, even when restricted to the classes of cubic graphs and bipartite graphs.
    
    \item[2)] We explore the relations between the SF-dimensions of a graph and the maximum number of leaves in its spanning trees. The latter defines a well-studied combinatorial problem Maximum Leaf Spanning Tree. \cite{galbiati1994short,griggs1989spanning,lu1998approximating,reich2016complexity}, which seems to be the closest to our problem. Indeed, both problems aim to find a ``star-like'' spanning tree; furthermore, several reduction schemes from the previous section exploit this relation. Given these observations, it may seem reasonable to try to adopt algorithmic machinery developed for the Maximum Leaf Spanning Tree problem. We prove the sharp upper bound for the  $s$-metric of a tree in terms of its number of leafs and diameter which, in conjunction with previously known similar lower bounds, reinforce such connections.  On the other hand, we present a family of counter-examples demonstrating that in general the difference between the SF-dimensions of a graph and its max-leaf spanning trees could be arbitrarily large. 
    
    \item[3)] We study in detail SF-dimension of split graphs~--- well-known class of graphs extensively used in both theory and applications \cite{golumbic2004algorithmic,mahadev1995threshold}. In particular, a number of generally NP-hard problems become polynomially solvable when restricted to split graphs \cite{brandstadt1999graph}.  Here  we establish sharp lower and upper bounds on the second SF-dimension and characterize the extremal graphs with respect to them. These results also imply the problem NP-hardness for split graphs, but its polynomial solvability in its subclass of threshold graphs.
    
    \item[4)] On the practical side, we propose two Integer Linear Programming formulations and two fast heuristics for the \textsc{$s$-SF Spanning Tree} problem, and perform computational experiments to assess their performance using simulated graphs and experimental graphs constructed from genomic data used for viral outbreaks investigation. The latter results are used to demonstrate how the concept of scale-free spanning tree could be useful in computational epidemiology.
\end{itemize}

\section{Notations, definitions and preliminary results}
\label{scale-free:sec2}

In this paper, we consider only finite, undirected graphs without loops and multiple edges. Also all graphs are assumed to be connected. We use graph-theoretic terminology of Chartrand et al.~\cite{chartrand2010graphs} (unless noted otherwise), and computational complexity terminology of Garey and Johnson~\cite{garey2002computers}. For concepts related to approximability, we follow Ausiello et al.~\cite{ausiello1999approximability}.

Let $G$ be a graph. The vertex set and the edge set of $G$ are denoted by $V(G)$ and $E(G)$, respectively. We denote by $|G|$ the \emph{order} of $G$ (i.e., $|G| = |V(G)|$). A \emph{clique} of $G$ is a set of pairwise adjacent vertices and an \emph{independent set} of $G$ is a set of pairwise nonadjacent vertices. A graph $H$ is a \emph{subgraph} of the graph $G$ if $V(H) \subseteq V(G)$ and $E(H) \subseteq E(G)$. If $V(H) = V(G)$, then $H$ is a \emph{spanning subgraph} of $G$. If two distinct vertices $u, v \in V(G)$ are adjacent, then the edge connecting them will be denoted by $uv$. The vertices $u$ and $v$ are called the \emph{end-vertices} of the edge $uv$. As usual, $N_G(x)$ denotes the \emph{neighborhood} of a vertex $x \in V(G)$, i.e., the set of all vertices that are adjacent to $x$ in $G$. If $y \in N_G(x)$, then $y$ is called a \emph{neighbor} of $x$ in $G$. The \emph{degree} of $x$ is defined as $\deg_G x = |N_G(x)|$. If the graph $G$ is clear from the context, we often omit the subscript $G$. A vertex of degree $0$ is referred to as an \emph{isolated} vertex and a vertex of degree $|G| - 1$ is a \emph{universal} vertex. A \emph{leaf} is a vertex of degree $1$. An edge incident with a leaf is called a \emph{pendant edge}. The maximum degree among the vertices of $G$ is denoted by $\Delta(G)$. 

A \emph{tree} is a connected acyclic graph. A \emph{spanning tree} of a graph $G$ is a spanning subgraph of $G$ that is a tree. We denote by $\ell(G)$ the maximum number of leaves in a spanning tree of $G$. A graph $G$ is called \emph{split} if its vertex set $V(G)$ can be partitioned into sets $K$ and $I$ such that $K$ is a clique and $I$ is an independent set. The complete graph, the path and the cycle on $n$ vertices are denoted by $K_n$, $P_n$ and $C_n$, respectively. A \emph{star} $K_{1, n}$ is the complete bipartite graph with partition classes of cardinalities $1$ and $n$. A \emph{double star} $S_{m, n}$ is the tree obtained from two disjoint stars $K_{1, m}$ and $K_{1, n}$ with $m$ and $n$ leaves, respectively, by adding an edge joining the central vertices of the two stars. For the purposes of Section~\ref{scale-free:sec4}, we will need the notion of a \emph{null graph} $K_0$ (in the terminology of Tutte~\cite{tutte84}), i.e., the graph having no edges and no vertices. 


Let $T$ be a tree. For a pair $(u, v)$ of distinct vertices $u, v \in V(T)$, let $P_T(u, v)$ be a unique path connecting $u$ and $v$ in $T$. We will denote by $u^{+}$ and $v^{-}$ the neighbors of $u$ and $v$ on $P_T(u, v)$, respectively.

The \emph{complement} $\overline{G}$ of a graph $G$ is the graph whose vertex set is $V(G)$ and where $e$ is an edge of $\overline{G}$ if and only if $e$ is not an edge of $G$. The \emph{corona} $G_1 \circ G_2$ of two graphs $G_1$ and $G_2$ is the graph obtained by taking one copy of $G_1$ and $n$ copies of $G_2$ (where $n$ is the order of $G_1$), and by joining each vertex of the $i$th copy of $G_2$ to the $i$th vertex of $G_1$, $i = 1, 2, \ldots, n$.

The invariants \emph{$s$-metric}, \emph{$m$-metric}, \emph{first SF-dimension} and \emph{second SF-dimension} of a graph $G$ are defined by expressions~\eqref{eq:smetric}, \eqref{eq:mmetric} and~\eqref{eq:sdim}, respectively. 
By $T^{\mathrm{sopt}}$ and $T^{\mathrm{mopt}}$ we denote an \emph{$s$-optimal tree} and an \emph{$m$-optimal tree} of $G$, respectively. Thus, we have $s(T^{\mathrm{sopt}}) = \tau_2(G)$ and $m(T^{\mathrm{mopt}}) = \tau_1(G)$. 

It is possible to provide lower and upper bounds for both SF-dimensions of a graph in terms of its order only. They follow from the bounds on first \cite{das2003sharp,de1998upper} and second \cite{das2004some} Zagreb indices of $n$-vertex trees derived in prior studies: 

\begin{proposition}[\cite{das2004some,das2003sharp,de1998upper}]
\label{kkos-theorem1}
For any tree $T$ of order $n \ge 3$,
$$4n - 6 \le m(T) \le n(n - 1),\qquad
4n - 8 \le s(T) \le (n - 1)^2.$$
Lower bounds are achieved if and only if $T \cong P_n$, and upper bounds are achieved whenever $T\cong K_{1,n-1}$.
\end{proposition}

This proposition directly implies the following corollary:

\begin{corollary}\label{kkos-theorem3}
For any graph $G$ of order $n \ge 3$,
$$4n - 6 \le \tau_1(G) \le n(n - 1),\qquad
4n - 8 \le \tau_2(G) \le (n - 1)^2,$$
with equalities for the lower bounds if and only if $G$ is isomorphic to $P_n$ or $C_n$, and equalities for the upper bounds if and only if $G$ has a universal vertex.
\end{corollary}

In the remaining part of this section, we introduce major proof techniques employed in this paper and prove several preliminary results.

\subsection{Path counting}
\label{pathcount:sec2.5}

This technique allows for efficient calculation of $m$-metric and $s$-metric and comparison of their values for structurally similar graphs. It is used to establish complexity results presented in Section \ref{scale-free:sec5}. The technique is based on the  
following expressions for the $m$-metric and $s$-metric in terms of numbers of trails of lengths at most 3:

\begin{proposition}\label{kkos-lemma3}	
For any graph $G$,
$$m(G) = 2\gamma_2(G) + 2\gamma_1(G), \qquad s(G) = \gamma_3(G) + 2\gamma_2(G) + \gamma_1(G),$$
where $\gamma_{t}(G)$ is the number of trails in $G$ with $t$ edges.
\end{proposition}
\begin{proof}
We prove only the second equality, the first one can be verified similarly.
Let $A$ be the adjacency matrix of $G$ and $\boldsymbol{d}$ be its degree vector. By the definition, $s(G) = \frac{1}{2} \boldsymbol{d}^T\cdot A \cdot \boldsymbol{d}$. For $\boldsymbol{d}$, in turn, we have $\boldsymbol{d} = A\cdot\boldsymbol{1}$, where $\boldsymbol{1} = (1,\ldots, 1)^T \in \mathbb{R}^n$. Therefore 
$$s(G) = \frac{1}{2} \boldsymbol{1}^T\cdot A^3 \cdot \boldsymbol{1} = \frac{1}{2}\sum\limits_{i = 1}^n \sum\limits_{j = 1}^n A^3_{i, j}.$$

It is well known, that $A^3_{i, j}$ is equal to the number of walks of length 3 between vertex $i$ and vertex $j$. Thus, $s(G)$ is equal to one-half of the total number of 3-walks in $G$.  An edge $\{v_1, v_2\}$ produces exactly two such walks: $W_1 = (v_1, v_2, v_1, v_2)$ and $W_2 = (v_2, v_1, v_2, v_1)$. Each 2-path $(v_1, v_2, v_3)$ produces four 3-walks: $W_1 = (v_1, v_2, v_3, v_2)$, $W_2 = (v_3, v_2, v_1, v_2)$, $W_3 = (v_2, v_1, v_2, v_3)$ and $W_4 = (v_2, v_3, v_2, v_1)$. Finally, each 3-path $(v_1, v_2, v_3, v_4)$ (with the possibility that $v_1 = v_4$) produces two 3-walks: $W_1 = (v_1, v_2, v_3, v_4)$ and $W_2 = (v_4, v_3, v_2, v_1)$. As every 3-walk of $G$ has one of these forms, the statement of the lemma follows. 
\end{proof}

\subsection{Neighbor switching}
\label{switch:sec2.5}

In this subsection we present a switching technique, introduced informally in~\cite{das2004some}, which is based on tree transformations and turned out to be a useful tool for obtaining structural and complexity results in our paper. 

Let $T$ be a tree and let $(u, v)$ be a pair of distinct vertices $u, v \in V(T)$ lying on the path $P_T(u, v)$, where $\deg_T u = p \ge 2$ and $\deg_T v = t \ge 2$. Let $A = N_T(u) \setminus \{u^{+}\} = \{a_1, \ldots, a_{p - 1}\}$, and the set $N_T(v) \setminus \{v^{-}\}$ is partitioned into two subsets $B = \{b_1, \ldots, b_q\}$ and $C = \{c_1, \ldots, c_r\}$, where $B \ne\emptyset$. Further, let $\deg_T u^{+} = \alpha$ and $\deg_T v^{-} = \beta$. Define numbers $D_A$, $D_B$ and $D_C$ as follows:
\begin{equation}\label{kkos-equ1}
D_A = \sum_{i = 1}^{p - 1} \deg_T a_i,\qquad	
D_B = \sum_{j = 1}^q \deg_T b_j,\qquad
D_C = \sum_{k = 1}^r \deg_T c_k.
\end{equation}


Now for the fixed pair $(u,v)$ we can perform the switching, i.e. a transformation producing a new tree $\widetilde{T}$ from $T$ as follows: we delete the edges
$vb_1, \ldots, vb_q$ and add new edges $ub_1, \ldots, ub_q$. In this case we say that $\widetilde{T}$ is produced
from the tree $T$ by the \emph{neighbor switch $\mathcal{S}_{v \rightarrow u}^B$} (or simply $\mathcal{S}_{v \rightarrow u}^B(T) = \widetilde{T}$).
The neighbor switch is illustrated in Fig.~\ref{kkos-figure2}. Note
that it changes only the degrees of the vertices $u$
and $v$, i.e.
$\deg_{\widetilde{T}} u = p + q$, $\deg_{\widetilde{T}} v = r + 1$, and
$\deg_{\widetilde{T}} x = \deg_T x$ for every vertex
$x \in V(T) \setminus \{u, v\}$.

\begin{figure}[h]
\centering
\includegraphics[width=11cm]{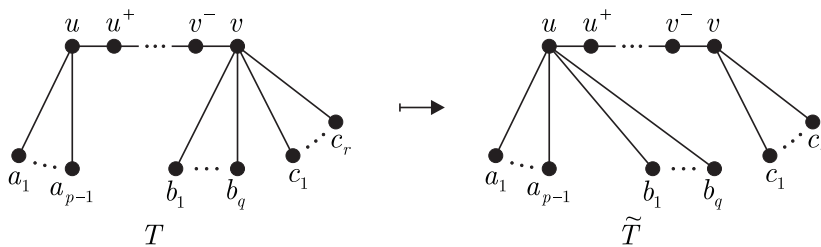}
\caption{An illustration of the neighbor switch}
\label{kkos-figure2}
\end{figure}

Taking into account definitions made above, we can prove the following lemma.

\begin{lemma}\label{kkos-lemma1}
Suppose that $\mathcal{S}_{v \rightarrow u}^B(T) = \widetilde{T}$. If $p \ge r + 1$, $D_A > D_C$ and additionally  $\alpha \ge \beta$, when $u$ and $v$ are not adjacent. 
Then $s(\widetilde{T}) > s(T)$.
\end{lemma}	

\begin{proof}
We provide the proof for the case when $u$ and $v$ are not adjacent, i.e. $u \neq v^{-}$ and $v \neq u^{+}$ (the opposite case can be verified similarly). 
Define by $X$ (resp., $Y$) the set of edges of $T$ (resp., $\widetilde{T}$)
incident to $u$ or $v$. Let us denote by $\lambda(X)$ the
contribution to $s(T)$ from the edges of $X$. Similarly, let
$\widetilde{\lambda}(Y)$ denote the contribution to $s(\widetilde{T})$ from
the edges of $Y$. Then we have
\begin{equation}\label{kkos-equ2}
s(\widetilde{T}) - s(T) = \widetilde{\lambda}(Y) - \lambda(X).
\end{equation}

Using~\eqref{kkos-equ1} one can easily calculate
\begin{equation*}
\begin{split}
\lambda(X) &= \deg_T u\deg_T u^{+} + \deg_T v^{-}\deg_T v +
\sum\limits_{i = 1}^{p - 1} \deg_T u \deg_T a_i +
\sum\limits_{j = 1}^q \deg_T v \deg_T b_j  \\
&+ \sum\limits_{k = 1}^r \deg_T v \deg_T c_k 
= p\alpha + \beta t + pD_A + tD_B + tD_C.
\end{split}
\end{equation*}
After substituting $t = q + r + 1$, we obtain
\begin{equation}\label{kkos-equ3}
\lambda(X) = p\alpha+ \beta q + \beta(r + 1) + pD_A + qD_B + (r + 1)D_B + qD_C + (r + 1)D_C.
\end{equation}
Similarly, 
\begin{equation}\label{kkos-equ4}
\widetilde{\lambda}(Y) = p\alpha + q\alpha + \beta(r + 1) + pD_A + qD_A + pD_B + qD_B + (r + 1)D_C.
\end{equation}

Using equalities~\eqref{kkos-equ2}--\eqref{kkos-equ4} we obtain
\begin{equation*}
\begin{split}
s(\widetilde{T}) - s(T)
&= \widetilde{\lambda}(Y) - \lambda(X) = q\alpha + qD_A + pD_B - \beta q - (r + 1)D_B - qD_C \\
&= q(\alpha - \beta) + D_B(p - r - 1) + q(D_A - D_C).
\end{split}
\end{equation*}
Since $\alpha \ge \beta$ and $p \ge r + 1$, it follows that
$q(\alpha - \beta) + D_B(p - r - 1) \ge 0$. On the other hand, since $q \ge 1$ and $D_A > D_C$,
it follows that $q(D_A - D_C) > 0$ and so $q(\alpha - \beta) + D_B(p - r - 1) + q(D_A - D_C) > 0$.
Therefore, $s(\widetilde{T}) - s(T) > 0$ and so $s(\widetilde{T}) > s(T)$,
producing the desired inequality.

\end{proof}

In particular, if $B = N_T(v)\setminus \{v^{-}\}$, then the neighbor switch produces a 
tree $\widetilde{T}$ with $v$ being a leaf. In
this case the transformation $\mathcal{S}_{v \rightarrow u}^B$ will be referred to as \emph{total neighbor switch}. For such transformation, since $D_A \ge p - 1 \ge 1$ (recall $\deg_T u = p \ge 2$) and $D_C = r = 0$, we have $D_A > D_C$ and $p \ge r+1$. It implies the following corollary.

\begin{corollary}\label{kkos-corollary1}
If $\widetilde{T}$ is obtained from $T$ by a total neighbor switch $\mathcal{S}_{v \rightarrow u}^B$, and additionally $\alpha \ge \beta$ when $u$ and $v$ are not adjacent, then $s(\widetilde{T}) > s(T)$.
\end{corollary}

The same way we can compare trees $T$ and  $\widetilde{T} = \mathcal{S}_{v \rightarrow u}^B(T)$ in terms of $m$-metric. Since only degrees of vertices $u$ and $v$ were changed by the neighbor switch, $m(\widetilde{T}) - m(T) = \deg_{\widetilde{T}}^2 u + \deg_{\widetilde{T}}^2 v - \deg_{T}^2 u - \deg_{T}^2 v = 2q(p - r - 1)$ which proves the next lemma, since $q \ge 1$.

\begin{lemma}\label{kkos-lemma2}
Suppose that $\mathcal{S}_{v \rightarrow u}^B(T) = \widetilde{T}$ and $p > r + 1$, then $m(\widetilde{T}) > m(T)$.
\end{lemma}	

For further results we need weaker modifications of Lemmas \ref{kkos-lemma1} and \ref{kkos-lemma2} for the case $\deg_T u = p \ge 1$ (and therefore $D_A \ge 0$). Recall $\deg_T v = t \ge 2$ since we still require at least one vertex to switch.

\begin{lemma}\label{kkos-corollary2}
Suppose $\widetilde{T}$ is obtained from $T$ by a total neighbor switch $\mathcal{S}_{v \rightarrow u}^B$,  then the following propositions hold:
\begin{itemize}
\item[a)] $m(\widetilde{T}) \ge m(T)$;
\item[b)] $s(\widetilde{T}) \ge s(T)$, if additionally $\alpha \ge \beta$ when $u$  and $v$  are not adjacent.
\end{itemize}
\end{lemma}

\section{Complexity and approximability results}
\label{scale-free:sec5}

In this section we study computational complexity of \textsc{$m$-SF Spanning Tree} and \textsc{$s$-SF Spanning Tree} problems. First we establish APX-hardness and NP-hardness of \textsc{$m$-SF Spanning Tree} and \textsc{$s$-SF Spanning Tree} respectively for cubic graphs. The rest of the section is dedicated to proving NP-hardness of both problems for bipartite graphs.

The following known fact will be used:

\begin{theorem}[\cite{kleitman1991spanning}]
\label{thm:mindegleaves}
Any connected graph of order $n$ with minimum vertex degree at least $3$ has a spanning tree with at least $n/4 + 2$ leaves.
\end{theorem}

Further let $G$ be a cubic graph on $n$ vertices and $T$ be a spanning tree with $\ell = \ell(T)$ leaves and $n_i = n_i(T)$ vertices of degree $i$, $i \in \{2, 3\}$. Then  
\begin{equation}
\label{eq:mcubic}
m(T) = \ell + 4n_2 + 9n_3,
\end{equation}
with the numbers $n_i$ satisfying the equalities $\ell + n_2 + n_3 = n$ and $\ell + 2n_2 + 3n_3 = 2(n-1)$.

Deriving $n_2$ and $n_3$ from these equalities gives us
\begin{equation}\label{eq:degdistrcubic}
n_2 = n + 2 - 2\ell,\qquad n_3 = \ell - 2. 
\end{equation}
After substituting these expressions into~\eqref{eq:mcubic} we get 
\begin{equation}\label{eq:mmetriccubic}
m(T) = 2\ell + 4n - 10. 
\end{equation}
Thus, finding a spanning tree with maximum $m$-metric in this case is equivalent to finding the spanning tree with the maximum number of leaves which is a known NP-hard Maximum Leaf Spanning Tree problem \cite{garey2002computers}, abbreviated as \textsc{MaxLeaf}.

\textsc{MaxLeaf}

\emph{Given}:  A connected graph $G$.

\emph{Find}: A spanning tree $T$ of $G$ with the maximum number of leaves $\ell(T)$.
\medskip

The \textsc{MaxLeaf} problem has been extensively studied. The main  results include its NP-hardness in a number of graph classes and approximability within a constant factor in general (see e.g. \cite{galbiati1994short,griggs1989spanning,lu1998approximating,reich2016complexity}). For cubic graphs this problem is known to be APX-hard~\cite{bonsma2012max}, which we exploit to prove APX-hardness of \textsc{$m$-SF Spanning Tree} by providing an L-reduction  \cite{papadimitriou91} from \textsc{MaxLeaf}. 

Given an optimization problem $P$ and an instance $I$ of this problem, we use $opt_P(I)$ to denote the optimum value of $I$, and $val_P(I, S)$ to denote the value of a feasible solution $S$ of instance $I$. Let $A$ and $B$ be two optimization problems.  Then $A$ is said to be L-reducible to $B$ if there exist polynomial-time computable functions $f$, $g$ and two constants $\alpha, \beta > 0$ such that 

\begin{itemize}
\item[(L1)] $f$ maps an instance $I$ of $A$ to an instance $f(I)$ of $B$ such that $opt_B (f(I)) \le \alpha \cdot opt_A(I)$ for all instances $I$ of $A$;
\item[(L2)] $g$ maps for any instance $I$ of $A$ a solution $S'$ for instance $f(I)$ of $B$ to a solution $S$ for $I$ such that $|val_A (I, S) - opt_A (I)| \le \beta \cdot |val_B (f (I), S') - opt_B (f(I))| $.
\end{itemize}

Let $T^{mopt}$ be an $m$-optimal spanning tree of $G$ and $\ell^*$ be the maximum number of leaves in spanning trees of $G$. Note $\ell^* \ge n/4 + 2$ by Theorem \ref{thm:mindegleaves} and therefore $n \le 4\ell^* - 8$. Then using~\eqref{eq:mmetriccubic} we get
$$
\tau_1(G) = m(T^{mopt}) \le 2 \ell (T^{mopt}) + 4n - 10 \le 2 \ell^* + 16\ell^* - 32 \le 18 \ell^*.
$$
Moreover, for every spanning tree $T$ of $G$ we have $\frac{1}{2} |m(T) - m(T^{mopt})| = |\ell(T) - \ell^*|$. As a result,~\eqref{eq:mmetriccubic} implies an L-reduction with identity mappings $f$ and $g$ and constants $\alpha = 18$ and $\beta = \frac{1}{2}$, proving the next theorem. 

\begin{theorem}\label{thm:apxhard}
The \textsc{$m$-SF Spanning Tree} problem is $\mathrm{APX}$-hard for cubic graphs.
\end{theorem}


 
Next we consider the \textsc{$s$-SF Spanning Tree} problem for cubic graphs. As above, let $G$ be a cubic graph on $n$ vertices and $T$ be a spanning tree of $G$.
\begin{theorem}
The \textsc{$s$-SF Spanning Tree} problem is $\mathrm{NP}$-hard for cubic graphs.
\end{theorem}

\begin{proof}
For the reduction, we will use the following problem proved to be NP-complete in~\cite{lemke1988maximum}:

\emph{Instance:} A connected cubic graph $G$.

\emph{Question:} Is there a spanning tree of $G$ without vertices of degree 2?
\medskip

According to~\eqref{eq:degdistrcubic}, $n_2 = n_2(T) = n + 2 - 2\ell(T)$. Thus the answer for the problem's question is negative if $n$ is odd. Hence we will concentrate only on the case when $n \ge 4$ is even, in which case $n_2$ is also even. We will show that among all $n$-vertex trees $T$ ($n \geq 4$ is even) with $\Delta(T) \leq 3$ the trees without vertices of degree 2 have the highest $s$-metric. Indeed,  the following claim holds:

\begin{claim}\label{claim:cubic2}
If $\Delta(T) \leq 3$ and $n \ge 4$ is even, then $s(T)\leq 6n - 15$. The equality holds if and only if $T$ has no vertices of degree $2$. 
\end{claim}
\begin{proof}
If $T$ has no vertices of degree 2, then~\eqref{eq:degdistrcubic} implies that $\ell = \ell(T) = \frac{n+2}{2}$. Furthermore, $s(T) = 3m_1 + 9m_3$, where $m_1$ is the number of pendant edges and $m_3$ is the number of edges with both ends of degree 3. Obviously, $m_1 = \ell$ and $m_3 = n - 1 - \ell$, thus yielding $s(T) = 6n - 15$.

Now suppose that $T$ has $n_2\geq 2$ vertices of degree $2$. Let $u$ and $v$ be two vertices of degree 2 lying on a path $P_T(u,v)$. Without loss of generality we may assume $\deg_T u^{+} \ge \deg_T v^{-}$. By iteratively repeating a total neighbor switch $\mathcal{S}_{v \rightarrow u}^B$ for all pairs of vertices $u$ and $v$ of degree 2, we will obtain a tree with higher $s$-metric (due to Corollary \ref{kkos-corollary1}) and without vertices of degree 2. This proves the claim.
\end{proof}

According to Claim~\ref{claim:cubic2}, $\tau_2(G) \leq 6n - 15$ for $n \ge 4$ is even, holds if and only if $G$ has a spanning tree without vertices of degree 2. This observation concludes the proof.
\end{proof}

Note that for cubic graphs, \textsc{$m$-SF Spanning Tree} and \textsc{$s$-SF Spanning Tree} problems are obviously approximable within a constant factor. The above claims allow to refine the approximation factors. In particular, the upper bound from Claim \ref{claim:cubic2} and the lower bound from Corollary \ref{kkos-theorem3} imply the existence of $\frac{3}{2}$-approximation for the \textsc{$s$-SF Spanning Tree} problem.

We proceed by proving that the scale-free spanning tree problems are NP-hard for bipartite graphs. We present a polynomial-time reduction from the \textsc{3-Dimensional Matching} problem, abbreviated as \textsc{3-DM}~\cite{garey2002computers}.

\textsc{3-DM}

\emph{Instance}: Pairwise disjoint sets $X$, $Y$, $Z$ each of cardinality $n$, and a collection $\mathcal{M}$ of $m$ three-element sets, where each member of $\mathcal{M}$ includes exactly one element from each of $X$, $Y$, and $Z$.

\emph{Question}: Is there a set of pairwise disjoint members of
$\mathcal{M}$, whose union is $X \cup Y \cup Z$?
\medskip

A set of pairwise disjoint members of $\mathcal{M}$, whose union is $X \cup Y \cup Z$, will be called a \emph{perfect $3$-dimensional matching}. Let $Q = (X, Y, Z, \mathcal{M})$ be an instance of 3-DM. For this instance we will construct a graph $G = G_Q$ on $3n + m + 1$
vertices as follows. The vertex set of $G$
consists of the disjoint union $\{r\} \cup A \cup B$ with the special root vertex
$r$, $A = \mathcal{M}$, and $B = X \cup Y \cup Z$. We introduce all the
edges $ra$ with $a \in A$ as well as, for each $a = M \in A$, the three
edges $ax$, $ay$, and $az$ where $M = \{x, y, z\}$. It is clear if $G$ is not
connected, then $\mathcal{M}$ contains no perfect $3$-dimensional matching. Therefore further we assume that $G$ is connected. Note also that $G$ is
bipartite graph with the parts $A$ and $\{r\} \cup B$. An example construction of $G$ is shown in Fig.~\ref{kkos-figure4}.

\begin{figure}[h]
\centering
\includegraphics[width=9cm]{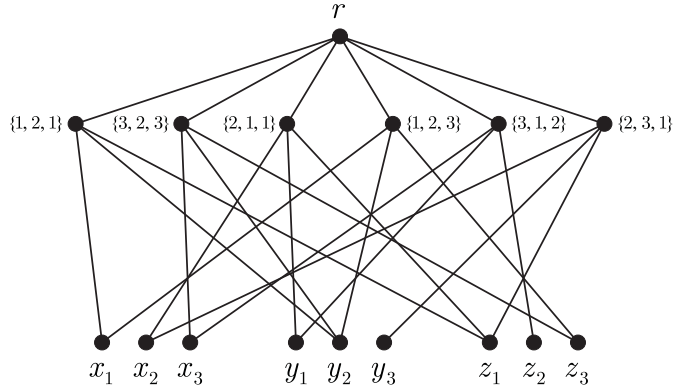}
\caption{An example of the graph $G$ for $n = 3$, $X = \{x_1, x_2, x_3\}$,
$Y = \{y_1, y_2, y_3\}$, $Z = \{z_1, z_2, z_3\}$, and
$\mathcal{M} = \{\{x_1, y_2, z_1\}, \{x_3, y_2, z_3\}, \{x_2, y_1, z_1\},
\{x_1, y_2, z_3\},\{x_3, y_1, z_2\}$, $\{x_2, y_3, z_1\}\}$. Here each vertex labelled $\{p, q, r\}$ represents a set $\{x_p, y_q, z_r\}$.}
\label{kkos-figure4}
\end{figure}

For a vertex $v$ of $G$ and a subset $W \subseteq V(G)$ let us
denote by $(v : W)$ the set of all edges connecting $v$ to vertices in $W$. 

\begin{lemma}\label{kkos-claim5}
There are a spanning trees $T_1$ and $T_2$ in $G$, both containing all edges of $(r : A)$, with $m(T_1) = \tau_1(G)$ and $s(T_2) = \tau_2(G)$.
\end{lemma}
	
\begin{proof}
We provide proof for $s(T_2) = \tau_2(G)$ only. The equality $m(T_1) = \tau_1(G)$ can be shown similarly. Among the spanning trees $T$ of $G$ with $s(T) = \tau_2(G)$, let $T_2^*$ be one that has the maximum number of edges from $(r : A)$. We claim
that $T_2^*$ contains all edges from $(r : A)$.
		
Suppose for a contradiction that the set $C \subseteq A$ of all vertices
that are adjacent to $r$ in $T_2^*$ is not equal to $A$. Then there would be a
vertex $b \in B$ adjacent in $T_2^*$ to some vertex $c$ in $C$, for which the
set $D$ of neighbors of $b$ in $T_2^*$ that are contained in $A \setminus C$
is non-empty.  By Lemma \ref{kkos-corollary2}, since $\deg_{T_2^*} r^{+} = \deg_{T_2^*} b^{-} = \deg_{T_2^*} c$ and $\deg_{T_2^*} r \ge 1$,  we can construct a spanning tree $T_2'$ from $T_2^*$ applying total neighbor switch $\mathcal{S}_{b \rightarrow r}^B$ with $s(T_2') \ge s(T_2^*)$ and the root $r$ having more neighbors in $T_2'$ than it has in $T_2^*$. 


\end{proof}
		
Any spanning tree $T$ of $G$ containing all edges of $(r : A)$ has $m + 3n$
paths of length one, $3n(m - 1)$ paths of length three (each of
the $3n$ edges of the tree connecting $A$ and $B$ induces exactly $m - 1$
such paths), and $m(m - 1)/2 + 3n$ paths of length two that are not formed by a pair of edges between $A$ and $B$. There are $3\delta_4 + \delta_3$
remaining paths of length two, where $\delta_i$ is the number of vertices in
$A$ that have degree $i$ in the tree. Indeed, a vertex $v\in A$ with
$j \in \{0, 1, 2, 3\}$ neighbors from $B$ in the tree contributes no such
path in case of $j \in \{0, 1\}$, one such path in case of $j = 2$, and
three such paths in case of $j = 3$. Thus by Proposition \ref{kkos-lemma3}
$$m(G) = m^2 + m + 12n + 6\delta_4 + 2\delta_3, \qquad s(T) = m^2 + 3mn + 6n + 6\delta_4 + 2\delta_3.$$
Since $|B| = 3n$, we have $3\delta_4 + 2\delta_3 \le 3n$ and $6\delta_4 + 2\delta_3 \le 6\delta_4 + 4\delta_3 \le 6n$. Hence,
$6\delta_4 + 2\delta_3 \le 6n$ with equality holding if and only if $\delta_3 = 0$ and $\delta_4 = n$.
		
A perfect 3-dimensional matching $\mathcal{M^*} = \{M_1, \ldots, M_n\}$ induces a spanning tree
$T_{\mathcal{M^*}}$ that contains all edges from $(r : A)$ and edges $ax, ay, az$ for each
$a = \{x, y, z\} \in \mathcal{M^*}$.
Fig.~\ref{kkos-figure4}). 
For this tree we have  $\delta_4 = n$ and
$$m(T) = m^2 + m + 18n := t_1(n, m), \qquad s(T) = m^2 + 3mn + 12n =: t_2(n, m).$$


Conversely, every spanning tree of $G$ that contains all edges from
$(r : A)$ and has $m$-metric equal to $t_1(n, m)$ or $s$-metric equal to $t_2(n, m)$ (and thus $\delta_4 = n$) arises from a
perfect 3-dimensional matching.
	
By Lemma~\ref{kkos-claim5}, the graph $G$ satisfies
$\tau_1(G) \ge t_1(n, m)$ (resp. $\tau_2(G) \ge t_2(n, m)$) if and only if there is a spanning tree $T$ of $G$
that contains all edges from $(r : A)$ and whose $m$-metric (resp. $s$-metric) is equal to 
$t_1(n, m)$ (resp. $t_2(n, m))$. The latter is true if and only if the instance $Q$ of \textsc{3-DM} has a perfect 3-dimensional matching. We have
established the following hardness result:

\begin{theorem}
\label{kkos-theorem7}	
The \textsc{$m$-SF Spanning Tree} and \textsc{$s$-SF Spanning Tree} problems are $\mathrm{NP}$-hard for bipartite graphs.
\end{theorem}

\section{Relations with maximum-leaf spanning trees}
\label{scale-free:sec3}

In this section we explore the relations between SF-spanning trees and maximum-leaf spanning trees of a graph. This is a direct continuation of the analysis from the previous section, where several reduction schemes exploit these relations.  The major result is the establishment of bounds for the $m$- and $s$-metrics of a tree depending on its number of nodes, number of leaves and diameter. 

In light of Proposition~\ref{kkos-theorem1} and the reduction scheme used to prove Theorem \ref{thm:apxhard}, one might think that an optimal tree should have a maximum or almost maximum possible number of leaves since intuitively a structure of an optimal tree should be ``star-like''. However, this simple intuition turns out to be somewhat misleading. In fact, the difference $\tau_2(G) - \max_{T \in ML(G)} s(T)$, where the maximum is taken over the set $ML(G)$ of all spanning trees of $G$ with the maximum number of leaves, can be arbitrarily large, as illustrated by the following example. For an integer $k \ge 2$, let $G_k$ be the graph of order $|G_k| = 2k + 4$ shown in Fig.~\ref{kkos-figure1} together with two of its spanning trees $T'$ (left) and $T''$ (right). The edges of $G_k$ not belonging to the corresponding spanning tree are dashed. It is easy to see that $T'$ is the only spanning tree of $G_k$ with the maximum number of leaves. One can show that $s(T') = (k + 2)|G_k|$ and $s(T'') = (k + 2)|G_k| + k$. Therefore, for every integer $k \ge 2$ we have $$\tau_2(G_k)\,\, - \max_{T \in ML(G_k)} s(T) \ge k.$$

\begin{figure}[h]
\centering
\includegraphics[width=9cm]{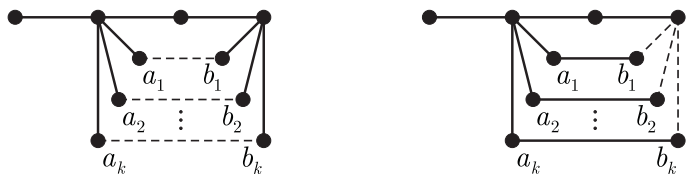}
\caption{Example of graph $G_k$ together with two of its spanning trees $T'$ (left) and $T''$ (right)}
\label{kkos-figure1}
\end{figure}

Nevertheless, within the class of trees $T$ there is a relation between the parameters $m(T)$, $s(T)$ and $\ell(T)$, which we will explore in the rest part of this section. The lower bounds for both Zagreb indices of a tree in terms of its number of leaves have been obtained previously and are summarized in the following theorem:

\begin{theorem}[\cite{goubko2014minimizing}]
For any tree $T$ with $\ell = \ell(T)$ leaves, the following statements hold:
\begin{itemize}
\item[a)] $m(T) \geq 9\ell - 16$;
\item[b)] if $\ell \geq 8$, then $s(T) \geq 11\ell - 27$.
\end{itemize}
Both bounds are sharp.
\end{theorem}

It is known that $m(T)\leq \frac{n}{n-1}s(T)$ holds~\cite{vukivcevic2007comparing}. Thus we have $\tau_1(G)\leq \frac{n}{n-1}\tau_2(G)$. In light of this fact, in the following we will establish upper bound in terms of the number of leaves just for the $s$-metric of a tree. We will use the following auxiliary definitions and properties. Let $T$ be a tree of diameter $d = \mathrm{diam}(T)$ and with $\ell = \ell(T)$ leaves.  A \emph{$2$-path} in $T$ is a maximal path with at least one internal node, all of whom  have  degree $2$. Among all 2-paths, we distinguish the paths with one  end vertex being a leaf. Such paths will be further referred to as \emph{pendant $2$-paths}, and the number of such paths will be denoted by $p_2$.

\begin{lemma}\label{lemma:auxleaf}
The following properties of a tree $T$ hold:
\begin{itemize}
\item[A1)] For each vertex $v$ in $T$, $\deg v \leq \ell$; and for every pair of vertices $v_1$ and $v_2$ in $T$, $\deg v_1 + \deg v_2 \leq \ell + 2$. 
\item[A2)] $p_2 \leq \ell$.
\item[A3)] Let $v$ be a leaf of $T$ adjacent to the vertex $u$, and $f(u) = \sum_{w\in N(u)} \deg w$. Then
\begin{equation}\label{eq:mnoleaf}
m(T) = m(T-v) + 2\deg u,
\end{equation}
\begin{equation}\label{eq:snoleaf}
s(T) = s(T-v) + f(u) + \deg u - 1.
\end{equation}
\item[A4)] Let $P^1$ and $P^2$ be $2$-paths in $T$. Then $|P^1| + |P^2| \leq d + 2$ and $|P^i| \leq d +1$, $i=1,2$.
\end{itemize}
\end{lemma}

\begin{proof}
The first part of statements A1) is implied by the following two facts: (i) every maximal path that starts at a neighbor of $v$ ends with a leaf; (ii) the paths that start at different neighbors of $v$ and do not contain $v$ are disjoint. The second part similarly follows from the following observations. Recall that $v_1^{+}$ and $v_2^{-}$ are the neighbors of $v_1$ and $v_2$ on the path $P_T(v_1, v_2)$. Then (i) every maximal path that starts at a vertex from the set $A = (N(v_1) \cup N(v_2)) \setminus \{v_1^{+},v_2^{-}\}$ ends with a leaf and (ii) the paths that start at different vertices of $A$ and and do not contain $v_1$ and $v_2$ are disjoint.



Statement A2) is implied by the fact that every pendant 2-path contains at least
one leaf and a leaf can be contained in at most one such path. Statement A3) could be directly verified using the definitions of $s(T)$ and $m(T)$. Finally,  statement A4) follows from the observation that any pair of 2-paths either do not intersect or have a common source vertex.
\end{proof}

\begin{theorem}\label{thm:maxleaf}
Let $T$ be a tree of order $n \ge 3$ having diameter $d$ and containing $\ell$ leaves. Then $s(T) \leq (d - 1)\ell^2$.
\end{theorem}

\begin{proof}
By Proposition \ref{kkos-theorem1}, the statement is true when $T \cong K_{1, n - 1}$ and $T \cong P_n$. If $d = 3$ then $T$ is isomorphic to either $P_4$ or a double star $S_{\ell_1, \ell_2}$ with $\ell = \ell_1 + \ell_2 \geq 3$. In this case it is easy to see that $$s(T) = \ell_1(\ell_1 + 1) + \ell_2(\ell_2 + 1) + (\ell_1 + 1)(\ell_2 + 1) \leq (\ell + 1)^2 \leq 2\ell^2 = (d-1)l^2.$$ 

If $d = 4$, then consider the central vertex $v$ of $T$ (i.e. the distance between $v$ and any other vertex of $T$ is at most 2). Suppose that this vertex is adjacent to $q$ leaves and $r$ non-leaf vertices, that are adjacent to $k_1, \ldots, k_r$ leaves, respectively. Then we have $$s(T) = q(q + r) + (q + r)\sum_{i = 1}^r (k_i + 1) + \sum_{i = 1}^r k_i(k_i + 1) = (q + r)^2 + (q + r + 1)\sum_{i = 1}^r k_i + \sum_{i = 1}^r k_i^2.$$

Suppose first that $k_i = 1$ for all $i = 1, \ldots, r$. Then $s(T) = (q + r)^2 + (q + r + 1)r + r$ and $\ell = q + r$. Given that $q + r \geq 2$, it is easy to see that $s(T) \leq 3(q + r)^2$, i.e. the statement of the theorem holds. Now assume that $\sum_{i = 1}^r k_i > 1$. In this case $q + r +1 \leq \ell$ and $\sum_{i = 1}^r k_i \leq \ell$. Then we have $s(T) \leq (q + r)^2 + (q + r + 1)\sum_{i = 1}^r k_i + (\sum_{i = 1}^r k_i)^2 \leq 3\ell^2$, i.e. the desired property holds again. 

So, further we assume that $5 \leq d \leq n - 2$ and $\ell \geq 3$.  For such trees we will prove the theorem using induction on the ordered pair $(d, n)$. Consider the following two cases.

1) \emph{There exists a path $P = (v_1, v_2, \ldots, v_{d + 1})$ of length $d$ which does not contain a pendant $2$-path.}

We have $\deg v_2 \geq 3$, $\deg v_d \geq 3$. The fact that $P$ has the maximum length implies that $\deg v_1 = \deg v_{d + 1} = 1$ and all neighbors of $v_2$ and $v_d$ are leaves with the exception of the vertices $v_3$, $v_{d - 1}$. Let $T' = T - \{v_1, v_{d + 1}\}$. The properties A1) and A3) imply that
\begin{equation}
\label{ineq:remleaf}
s(T) = s(T') + 2(\deg v_2 - 1) + \deg v_3 + 2(\deg v_d - 1) + \deg v_{d - 1} \leq  s(T') +  3\ell + 2.
\end{equation}
Furthermore, $\ell(T') = \ell - 2$, $|T'| = n - 2$ and $\mathrm{diam}(T') \leq d$. By utilizing the inductive hypothesis, we get
\begin{equation}
s(T) \leq (d - 1)(\ell - 2)^2 + 3\ell + 2 \leq (d - 1)\ell^2.
\end{equation}

2) \emph{All maximum paths of $T$ contain pendant $2$-paths.}

Suppose that $P = (v_1, v_2, \ldots, v_k, w)$ is a pendant 2-path, with $v_1$ being a leaf. Since $G$ is not isomorphic to $P_n$, we have $\deg w \geq 3$. By iteratively removing vertices $v_1, \ldots, v_{k - 1}$ and applying~\eqref{eq:snoleaf}, we get that $$s(T) = s(T - \{v_1, \ldots, v_{k - 1}\}) + 4(k - 2) + \deg w + 2.$$

Let $P^1, \ldots, P^{p_2}$ be the pendant 2-paths of $T$ ordered in decreasing order of their lengths. Note that by the property A4) $|P^i|\leq \frac{d}{2} + 1$ for all $i \geq 2$. Denote by $T'$ the tree obtained from $T$ by removal of vertices of these 2-paths, as described above. Let $w_i$ be the non-leaf starting vertex of the $i$th path. Using the properties A1), A4), we get 
\begin{equation}
\label{ineq:snoallpaths}
s(T) \leq s(T') + 4\sum_{i = 1}^{p_2} (|P^i| - 3) + 2p_2 + \sum_{i = 1}^{p_2} \deg w_i \leq s(T') + \rho(T), 
\end{equation}
where 
\begin{equation}\label{eq:rho}
\rho(T) = 
\begin{cases}
	4(d + 1) - 10 + \ell, \quad \text{if}\,\, p_2 = 1; \\
	4(d + 2)p_2/2 - 10p_2 + \ell p_2, \quad \text{if}\,\, p_2\,\, \text{is even}; \\
	4(d + 2)(p_2-1)/2 + d/2 + 1 - 10p_2 + \ell p_2, \quad \text{if}\,\, p_2 \geq 3 \,\, \text{is odd}.
\end{cases} 
\end{equation}

For the tree $T'$, there are no pendant 2-paths, $\ell(T') = \ell$, $|T'| \le n - 1$ and $\mathrm{diam}(T') \leq d - 1$. As in the case 1), consider the longest path $P = (v_1, v_2, \ldots, v_{d})$ in $T'$. The same reasoning as above yields 
\begin{equation}
\label{ineq:remleaf1}
s(T') \leq s(T'') +  3\ell + 2,
\end{equation}
where $T'' = T'-\{v_1, v_{d}\}$.
Furthermore, $\ell(T'') = \ell - 2$, $|T''| \le n-3 $ and $\mathrm{diam}(T'') \leq d - 1$. 

Consider the case when $p_2$ is even (other cases can be handled similarly). Using simple arithmetic transformations and the property A2) we get that $\rho(T) \leq (2d + \ell - 6) \ell$. By utilizing the inductive hypothesis and using~\eqref{ineq:snoallpaths},~\eqref{ineq:remleaf1} we get
\begin{equation*}
s(T) \leq (d - 2)(\ell - 2)^2 + 3\ell + 2 + \rho(T) \leq (d - 2)(\ell - 2)^2 + 3\ell + 2 + (2d + \ell -6) \ell.
\end{equation*}

\noindent
Given that $d\geq 5$ and $\ell \geq 3$, the right-hand side of this inequality does not exceed $(d - 1)\ell^2$. This proves the theorem.
\end{proof}

Note that the upper bound provided by Theorem~\ref{thm:maxleaf} is sharp, as it holds with equality for both $T = K_{1, n - 1}$ and $T = P_n$.

\section{Split graphs}
\label{scale-free:sec4}

In this section, we study structural properties of optimal trees of a split
graph. Based on these properties, for any split graph $G$, we establish
sharp lower and upper bounds for the second SF-dimension of $G$, characterize the extremal graphs with respect to them and
establish the computational complexity of \textsc{$s$-SF Spanning Tree} and \textsc{$m$-SF Spanning Tree} problems in the class of split graphs. Recall that a graph $G$ is called a \emph{split graph} if its vertex set $V(G)$ can be
partitioned into sets $K$ and $I$ such that $K$ is a clique and $I$ is an
independent set, where $(K, I)$ is called a \emph{split partition} of $G$.
A typical subclass of split graphs is the class of threshold graphs. A split
graph $G$ with a split partition $(K, I)$ is called a \emph{threshold graph} if
there exists an ordering $x_1, x_2, \ldots, x_{|I|}$ of the vertices in $I$
such that $N(x_1) \subseteq N(x_2) \subseteq \ldots \subseteq N(x_{|I|})$. The classes of split graphs and threshold graphs were introduced,
respectively, by F{\"o}ldes and Hammer~\cite{foldes1977split}, and Chv\'{a}tal
and Hammer~\cite{chvatal1977aggregations}, and have been extensively studied \cite{golumbic2004algorithmic,mahadev1995threshold}.

We say that a family $\mathcal{F}$ of graphs is \emph{closed under the
adjunction of universal} (resp., \emph{isolated}) \emph{vertices} if for every graph $G$ in $\mathcal{F}$, adjoining a new vertex
adjacent to all (resp., no) old vertices in $G$ produces another graph in
$\mathcal{F}$. Split graphs and threshold graphs are closed under the
adjunction of both universal and isolated vertices.

A well-known structural characterization of threshold graphs due to
Chv\'{a}tal and Hammer~\cite{chvatal1977aggregations} is the following: $G$ is a threshold graph if and only if $G$ can be built from the null graph $K_0$ by a
sequence of adjunctions of universal or isolated vertices. Consequently, in
a connected threshold graph $G$ there always exists at least one universal 
vertex and hence $\tau_2(G) = (n - 1)^2$ by Corollary~\ref{kkos-theorem3}.

In order to establish bounds for the second SF-dimensions of a split graph (i.e., Theorem~\ref{kkos-theorem5}), we first study the structural properties of $s$-optimal trees of split graphs. Thus, we let $G$ be a connected split graph with a split partition $(K, I)$ and $T^{\mathrm{sopt}}$ be an $s$-optimal tree of $G$, i.e., $\tau_2(G) = s(T^{\mathrm{sopt}})$. If $|K| = 1$, then $G$ is $K_{1, n - 1}$ and $\tau_2(G) = (n - 1)^2$. 
We see $\tau_2(K_n) = (n - 1)^2$, and without loss of generality, we may assume that $I$ is a non-empty and $K$ is a maximal clique. If $|K| = 2$, then $G$ is isomorphic to a double star $S_{m, n}$. An easy direct check shows that $\tau_2(S_{m, n}) = (m + n + 1)^2 - mn$. Therefore, we may further assume that $|K| \ge 3$.

We proceed with a series of claims. In the following proofs we are referring to the case of (total) neighbor switch with respect to a pair of adjacent vertices.  

\begin{claim}\label{kkos-claim1}
All vertices in $I$ are leaves of $T^{\mathrm{sopt}}$.
\end{claim}

\begin{proof}
Suppose that the statement is false. Then there exists some vertex $v \in I$
such that $\deg_{T^{\mathrm{sopt}}} v = t \ge 2$. Denote the neighbors of $v$
in $T^{\mathrm{sopt}}$ by $u, b_1, \ldots, b_{t - 1}$. Since $K$ is maximal
clique, it follows that $t < |K|$. This fact together with the connectivity
of $T^{\mathrm{sopt}}$ implies that at least one of the vertices
$u, b_1, \ldots, b_{t - 1}$ must have degree at least 2 in
$T^{\mathrm{sopt}}$. We may assume, without loss of generality, that
$\deg_{T^{\mathrm{sopt}}} u \ge 2$. Let tree $\widetilde{T}$ be obtained from
$T^{\mathrm{sopt}}$ by the total neighbor switch (i.e., by deleting the edges $vb_1, \ldots, vb_{t - 1}$ and
adding the edges $ub_1, \ldots, ub_{t - 1}$). Since $ub_i \in E(G)$ for
$i = 1, \ldots, t - 1$, it follows that $\widetilde{T}$ is a spanning tree
of $G$ and so $s(\widetilde{T}) > s(T^{\mathrm{sopt}})$ due to
Corollary~\ref{kkos-corollary1}. This, however, contradicts the optimality
of $T^{\mathrm{sopt}}$.
\end{proof}

Denote by $T^*$ the subtree obtained from $T^{\mathrm{sopt}}$ by deleting all
the leaves in $I$. For a vertex $x \in V(T^*)$, we will use $S_x$ to denote
the set of all vertices from $I$ which are adjacent to $x$ in
$T^{\mathrm{sopt}}$. By Claim~\ref{kkos-claim1}, $S_x \cap S_y = \emptyset$
for any two vertices $x$ and $y$ of $T^*$.

\begin{claim}\label{kkos-claim2}
The tree $T^*$ is a star.
\end{claim}

\begin{proof}
Assume, to the contrary, that $T^*$ is not a star. Then there is some edge
$uv$ of $T^*$ that joins two vertices $u$ and $v$ for which
$\deg_{T^*} u = p \ge 2$ and $\deg_{T^*} v = t \ge 2$. Without loss of generality,
we may assume that $|S_u| \ge |S_v|$. Also we let
$N_{T^*}(u) \setminus \{v\} = \{a_1, \ldots, a_m\}$ and
$S_u = \{a_{m + 1}, \ldots, a_{p - 1}\}$,
where $p = \deg_{T^{\mathrm{sopt}}} u$, and we note that $m \ge 1$. Partition
the set $N_{T^{\mathrm{sopt}}}(v) \setminus \{u\}$ into two subsets
$B = N_{T^*}(v) \setminus \{u\} = \{b_1, \ldots, b_q\}$ and
$S_v = \{c_1, \ldots, c_r\}$. Note that $q \ge 1$ and $r \ge 0$. Then for
the numbers $D_A$ and $D_C$ associated with $T^{\mathrm{sopt}}$ and defined
by~\eqref{kkos-equ1} we have
$$D_A = \sum\limits_{i = 1}^m \deg_{T^{\mathrm{sopt}}} a_i +
\sum\limits_{i = m + 1}^{p - 1} \deg_{T^{\mathrm{sopt}}} a_i \ge
m + |S_u| \ge 1 + |S_u| > |S_v| =
\sum\limits_{k = 1}^r \deg_{T^{\mathrm{sopt}}} c_k = D_C,$$
i.e., $D_A > D_C$. On the other hand, since $|S_u| = p - 1 - m$, $|S_v| = r$ and
$|S_u| \ge |S_v|$, it follows that $p \ge r + 1 + m$ and so $p > r + 1$
(since $m \ge 1$). Thus, all the conditions of Lemma~\ref{kkos-lemma1} hold,
implying the existence of a spanning tree $\widetilde{T}$ of $G$ such that
$s(\widetilde{T}) > s(T^{\mathrm{sopt}})$; the tree
$\widetilde{T}$ is obtained from $T^{\mathrm{sopt}}$ by the neighbor switch $\mathcal{S}_{v \rightarrow u}^B$, i.e., by deleting the edges
$vb_1, \ldots, vb_q$ and adding the edges $ub_1, \ldots, ub_q$ (notice that
$ub_i \in E(G)$ for $i = 1, \ldots, q$). This, however, contradicts the
optimality of $T^{\mathrm{sopt}}$. Thus, as claimed, $T^*$ is a star.
\end{proof}

\begin{claim}\label{kkos-claim3}
The central vertex of $T^*$ has the maximum number of neighbors from $I$ in
$T^{\mathrm{sopt}}$.
\end{claim}

\begin{proof}
By Claim~\ref{kkos-claim2}, the tree $T^*$ is a star. Since $|T^*| = |K|$
and $|K| \ge 3$, we let $v$ be the unique central vertex of $T^*$. Assume,
to the contrary, that there exists a vertex $u$ of $T^*$ distinct from $v$
such that $|S_u| > |S_v|$ hold. Then
$p = \deg_{T^{\mathrm{sopt}}} u \ge 2$. As in Claim~\ref{kkos-claim2}, we let
$N_{T^\mathrm{sopt}}(u) \setminus \{v\} = \{a_1, \ldots, a_{p - 1}\}$ and
partition the set $N_{T^{\mathrm{sopt}}}(v) \setminus \{u\}$ into two subsets
$B = N_{T^*}(v) \setminus \{u\} = \{b_1, \ldots, b_q\}$ and
$S_v = \{c_1, \ldots, c_r\}$. Note that $q \ge 1$, since $v$ is the central
vertex of the star $T^*$ and $|T^*| \ge 3$. Now we have
$$D_A = \sum\limits_{i = 1}^{p - 1} \deg_{T^{\mathrm{sopt}}} a_i = |S_u| >
|S_v| = \sum\limits_{k = 1}^r \deg_{T^{\mathrm{sopt}}} c_k = D_C,$$
i.e., $D_A > D_C$. On the other hand, since $|S_u| = p - 1$, $|S_v| = r$ and
$|S_u| > |S_v|$, it follows that $p > r + 1$. Thus, all the conditions of
Lemma~\ref{kkos-lemma1} are satisfied, implying (as in
Claim~\ref{kkos-claim2}) the existence of a spanning tree $\widetilde{T}$ of
$G$ such that $s(\widetilde{T}) > s(T^{\mathrm{sopt}})$, which is
impossible. Therefore, $|S_v| \ge |S_u|$ for each vertex $u$ of $T^*$.
\end{proof}

The central vertex of $T^*$ will be called the \emph{source} vertex of $T^{\mathrm{sopt}}$.

\begin{claim}\label{kkos-claim4}
The source vertex of $T^{\mathrm{sopt}}$ has degree $\Delta(G)$ in $T^{\mathrm{sopt}}$.
\end{claim}

\begin{proof}
Let $x^* \in V(T^*)$ be
the source vertex of $T^{\mathrm{sopt}}$. From Claim~\ref{kkos-claim3} we
know that $|S_{x^*}| \ge |S_x|$ for each vertex $x \in V(T^*)$. If we
assume that there exist vertices $y \in V(T^*) \setminus \{x^*\}$ and
$z \in S_y$ such that $x^*z \in E(G)$, then we can again apply
Lemma~\ref{kkos-lemma1} to construct a spanning tree $\widetilde{T}$ of $G$
such that $s(\widetilde{T}) > s(T^{\mathrm{sopt}})$ by deleting
the edge $yz$ of $T^{\mathrm{sopt}}$ and adding the edge $x^*z$. This
contradiction leads to the conclusion that $x^*$ has degree $\Delta(G)$ in
$T^{\mathrm{sopt}}$.
\end{proof}

Recall that we have assumed, without loss of generality, that $K$ is a
maximal clique of a split graph $G$, i.e., $I$ does not contain a vertex adjacent to all vertices of $K$.
This means that the split partition $(K, I)$ of $G$ is chosen to maximize
$|K|$, and consequently, $|K| = \omega(G)$, where $\omega(G)$ is the \emph{clique
number} of the graph $G$, i.e., the cardinality of a maximum clique of $G$.

We are now in a position to prove the main result of this section.

\begin{theorem}\label{kkos-theorem5}
If $G$ is a split graph of order $n$ having maximum degree
$\Delta(G) = \Delta$ and clique number $\omega(G) = \omega$, then
$$
\max\{4n - 8,\, 2n + (\Delta - 1)^2 - 3\} \le \tau_2(G) \le \min\{(n - 1)^2,\,
(\Delta - \omega + 2)(n + \Delta(\omega - 1) - 1) - \Delta\}.
$$
\end{theorem}

\begin{proof}
Let $T^{\mathrm{sopt}}$ be an $s$-optimal tree of $G$ and let
$x^* \in V(T^*)$ be the source vertex of $T^{\mathrm{sopt}}$. By
Claims~\ref{kkos-claim2} and~\ref{kkos-claim4}, the vertex $x^*$ has exactly
$\Delta - \omega + 1$ neighbors from $I$ in $T^{\mathrm{sopt}}$.
Since by Claim~\ref{kkos-claim1} all the vertices of $I$ are leaves in
$T^{\mathrm{sopt}}$, it follows that
\begin{equation}\label{kkos-equ8}
\sum\limits_{x^*y\, \in\, E(T^{\mathrm{sopt}})} \deg_{T^{\mathrm{sopt}}} x^*
\deg_{T^{\mathrm{sopt}}} y = \Delta(\Delta - \omega + 1),
\end{equation}
where the vertex $y$ in the subscript of the sum runs over the set
$S_{x^*}$.

Let $z$ be any of the remaining
$|I| - (\Delta - \omega + 1) = n - \Delta - 1$ leaves of
$T^{\mathrm{sopt}}$ in $I$ and let $z \in S_x$ for some vertex
$x \in V(T^*) \setminus \{x^*\}$. Obviously, the degree of $x$ is at least 2
in $T^{\mathrm{sopt}}$. On the other hand, this degree does not exceed
$\Delta - \omega + 2$, since otherwise $|S_x| > |S_{x^*}|$, which is
impossible by Claim~\ref{kkos-claim3}. Hence,
\begin{equation}\label{kkos-equ9}
2(n - \Delta - 1) \le \sum\limits_{xz\, \in\, E(T^{\mathrm{sopt}})}
\deg_{T^{\mathrm{sopt}}} x \deg_{T^{\mathrm{sopt}}} z \le
(\Delta - \omega + 2)(n - \Delta - 1),
\end{equation}
where the vertices $x$ and $z$ in the subscript of the sum run over the sets
$V(T^*) \setminus \{x^*\}$ and $S_x$ respectively.

Now let $x$ be any of the $\omega - 1$ vertices in
$V(T^*) \setminus \{x^*\}$. Note that $x^*x \in E(T^{\mathrm{sopt}})$, since
$T^*$ is a star due to Claim~\ref{kkos-claim2}. Thus, the degree of $x$ is
at least 1 in $T^{\mathrm{sopt}}$. On the other hand, as we saw above, this
degree does not exceed $\Delta - \omega + 2$. Hence,
\begin{equation}\label{kkos-equ10}
\Delta(\omega - 1) \le \sum\limits_{x^*x\, \in\, E(T^{\mathrm{sopt}})}
\deg_{T^{\mathrm{sopt}}} x^* \deg_{T^{\mathrm{sopt}}} x \le
\Delta(\Delta - \omega + 2)(\omega - 1),
\end{equation}
where the vertex $x$ in the subscript of the sum runs over the set
$V(T^*) \setminus \{x^*\}$.

Summation of~\eqref{kkos-equ8},~\eqref{kkos-equ9} and~\eqref{kkos-equ10},
upon little simplification, yields the following inequalities for $\tau_2(G)$:
$$
2n + (\Delta - 1)^2 - 3 \le \tau_2(G) \le
(\Delta - \omega + 2)(n + \Delta(\omega - 1) - 1) - \Delta.
$$
The final result now follows by applying Corollary~\ref{kkos-theorem3}.
\end{proof}

The following result characterizes connected split graphs for which the
upper and lower bounds for $\tau_2(G)$ in Theorem~\ref{kkos-theorem5} are
achieved.

\begin{theorem}\label{kkos-theorem6}
Let $G$ be a connected split graph of order $n$ having maximum degree
$\Delta(G) = \Delta$ and clique number $\omega(G) = \omega$. Then
\begin{itemize}
\item[$(\mathrm{i})$] $\tau_2(G) = \min\{(n - 1)^2,\,
(\Delta - \omega + 2)(n + \Delta(\omega - 1) - 1) -
\Delta\}$ if and only if one of the following conditions holds:
\begin{itemize}
\item[$(\mathrm{a})$] $n = \omega(t + 1)$ and
$G = K_{\omega} \circ \overline{K}_t$ for some integers $\omega \ge 1$ and
$t \ge 0$;
\item[$(\mathrm{b})$] $G$ has a universal vertex.
\end{itemize}
\item[$(\mathrm{ii})$] $\tau_2(G) = \max\{4n - 8,\, 2n + (\Delta - 1)^2 - 3\}$
if and only one of the following conditions holds:
\begin{itemize}
\item[$(\mathrm{c})$] $G = P_4$;
\item[$(\mathrm{d})$] $G$ has a universal vertex.
\end{itemize}
\end{itemize}
\end{theorem}

\begin{proof}
(i) The sufficiency part follows immediately by an easy direct calculation
of $\tau_2(G)$ for the graphs that satisfy the conditions (a) or (b).

Now we prove the necessity part of (i). If the minimum is $(n - 1)^2$, i.e.,
$\tau_2(G) = (n - 1)^2$, then by Corollary~\ref{kkos-theorem3}, $G$ contains a
universal vertex. Thus, $\Delta = n - 1$ and taking into account that
$n \ge \omega$ and $\omega \ge 1$, we have
$$(\Delta - \omega + 2)(n + \Delta(\omega - 1) - 1) - \Delta =
(n - 1)^2 + (n - 1)(\omega - 1)(n - \omega) \ge (n - 1)^2,$$
which is correct in the case when the minimum is equal to $(n - 1)^2$.
Therefore, the condition (b) holds.

Let the minimum is equal to
$(\Delta - \omega + 2)(n + \Delta(\omega - 1) - 1) - \Delta$, i.e.,
$$\tau_2(G) = (\Delta - \omega + 2)(n + \Delta(\omega - 1) - 1) - \Delta.$$
We may assume, without loss of generality, that $G$ contains no universal 
vertices and $|K| \ge 3$ (since if $|K| = 1$ or 2, then $G$ is a star
$K_{1, n - 1}$ or a double star $K_2 \circ\overline{K}_t$, where
$t = (n - 2)/2$, respectively). Let $T^{\mathrm{sopt}}$ be an $s$-optimal
tree of $G$ and let $x^* \in K$ (note also that $K = V(T^*)$)
be the source vertex of $T^{\mathrm{sopt}}$. By Claim~\ref{kkos-claim4},
$\deg_{T^\mathrm{sopt}} x^* = \deg_G x^* = \Delta$, and consequently
$|S_{x^*}| = \Delta - \omega + 1$. Moreover, from the proof of
Theorem~\ref{kkos-theorem5} we infer that $|S_x| = \Delta - \omega + 1$ for
each vertex $x \in K\setminus\{x^*\}$. Hence,
$\deg_G x = |S_x| + \omega - 1 = \Delta$ for each such vertex. Besides,
$S_x \cap S_y = \emptyset$ for any two vertices $x$ and $y$ in $K$, since
all vertices in $I$ are leaves of $T^{\mathrm{sopt}}$ by
Claim~\ref{kkos-claim1}. Therefore, there exists a partition
$\cup_{x \in K} S_x$ of $I$ such that $|S_x| = \Delta - \omega + 1$ for
each vertex $x \in K$. Note that there is no edge of $G$ connecting a vertex
$x$ in $K$ to a vertex in $S_y$ for any two distinct vertices $x, y \in K$,
since otherwise $\deg_G x > \Delta$, which is impossible. Thus, we have
$G = K_{\omega} \circ \overline{K}_t$, where $t = \Delta - \omega + 1$ and
$n = \omega(t + 1)$, and $G$ satisfies the condition (a).

(ii) As above, the sufficiency part follows immediately by an easy direct
calculation of $\tau$ for the graphs that satisfy the conditions (c) or (d).

Let us prove the necessity part of (ii). If the maximum is $4n - 8$, i.e.,
$\tau_2(G) = 4n - 8$, then by Corollary~\ref{kkos-theorem3}, $G$ is either $P_n$
or $C_n$ for $n \ge 3$. Since $G$ is a split graph, $G \in \{P_4, C_3\}$ and
so the conditions (c) or (d) hold. At the same time, since $\Delta = 2$ and
$n \ge 3$, we have $$2n + (\Delta - 1)^2 - 3 = 2n - 2 \le 4n - 8,$$ which is
correct in the case when the maximum is equal to $4n - 8$.

Now let the maximum is equal to $2n + (\Delta - 1)^2 - 3$, i.e.,
$\tau_2(G) = 2n + (\Delta - 1)^2 - 3$. We may assume, without loss of
generality, that $\Delta > 0$, since if $\Delta = 0$, then $\tau_2(G) = 0$ and
$G$ is $K_1$ which satisfies the condition (d). In the same manner we can
assume that $n \ge 3$. Let $T^{\mathrm{sopt}}$ be an $s$-optimal tree of
$G$ and let $x^* \in K$ be the source vertex of
$T^{\mathrm{sopt}}$. By Claim~\ref{kkos-claim4},
$\deg_{T^\mathrm{sopt}} x^* = \deg_G x^* = \Delta$. We show that $x^*$ is a
universal vertex of $G$. Assume, to the contrary, that $x^*$ is not a
universal vertex. Then there is a vertex $x \in K\setminus \{x^*\}$ that is
adjacent to some vertex $z \in S_x$ and $x^*z \not\in E(G)$. But then from
the proof of Theorem~\ref{kkos-theorem5} the left hand-side of~\eqref{kkos-equ10} would be $2\Delta + (\omega - 2)\Delta$ and the sum of~\eqref{kkos-equ8}--~\eqref{kkos-equ10} would imply
\begin{equation*}
\begin{split}
\tau_2(G)
&\ge \Delta(\Delta - \omega + 1) + 2(n - \Delta - 1) + 2\Delta +
(\omega - 2)\Delta \\
&= 2n + (\Delta - 1)^2 - 3 + \Delta > 2n + (\Delta - 1)^2 - 3,
\end{split}
\end{equation*}
which is contradiction. Thus, $x^*$ is a universal vertex of graph
$G$. Consequently, $G$ satisfies the condition (d). For completeness we
note that in this case
$$2n + (\Delta - 1)^2 - 3 = (n - 1)^2 \ge 4n - 8,$$
since $\Delta = n - 1$ and $n \ge 3$.
\end{proof}

The obtained structural characterization can be used to establish the complexities of \textsc{$s$-SF Spanning Tree} and \textsc{$m$-SF Spanning Tree} problems when restricted to split graphs. First note that proofs of Claims \ref{kkos-claim1}-\ref{kkos-claim4} rely on a neighbor switch, satisfying $p > r + 1$ in each particular case. Therefore Lemma \ref{kkos-lemma2} implies the following corollary. 
\begin{corollary}\label{kkos-cor3}
Claims \ref{kkos-claim1}~--~\ref{kkos-claim4} similarly hold for an $m$-optimal tree of a split graph $G$.
\end{corollary}

\begin{theorem}\label{kkos-theorem8}	
The \textsc{$s$-SF Spanning Tree} and \textsc{$m$-SF Spanning Tree} problems are $\mathrm{NP}$-hard for split graphs.
\end{theorem}

\begin{proof}
We will utilize the construction used to prove Theorem \ref{kkos-theorem7}.
We obtain a graph $H = H_Q$ by adding all edges $a_ia_j$ with $a_i,a_j \in A$,
$i \ne j$, to the graph $G_Q$ constructed from an instance $Q$ of \textsc{3-DM}. It can be easily observed that the vertex set of the resulting graph $H$ can be partitioned into the clique $K = \{r\} \cup A$ and
the independent set $I = B$, i.e., $H$ is a split graph. Thus we can exploit
results on the structure of its $s$-optimal tree $T^{\mathrm{sopt}}$ (resp. $m$-optimal tree $T^{\mathrm{mopt}}$). In particular, due to Claim~\ref{kkos-claim4}
one of the vertices $a_i$ is a source vertex of $T^{\mathrm{sopt}}$ (resp. $T^{\mathrm{mopt}}$), since the condition
$\deg_H v = \Delta(H) = m + 3$ holds only for vertices from $A$, and all
vertices in $B$ are leaves of $T^{\mathrm{sopt}}$ (resp. $T^{\mathrm{mopt}}$) due to Claim~\ref{kkos-claim1}.
	
Any $s$-optimal tree $T^{\mathrm{sopt}}$ (resp. $m$-optimal tree $T^{\mathrm{mopt}}$) of the constructed split graph $H$ clearly has $m + 3n$ paths of length one. Each of $3n - 3$ edges connecting $A$ and $B$ except for three edges incident to the source vertex, induces $m + 2$ paths of
length three. Additionally there exist $(m + 3)(m + 2)/2 - 3 + 3n - 3$ paths
of length two that do not consist of two edges connecting $A$ and $B$. There are $3\delta_4 + \delta_3$ remaining paths of length
two, where $\delta_i$ is again the number of vertices in $A$ that have degree $i$ in the tree $T^{\mathrm{sopt}}$ (resp. $T^{\mathrm{mopt}}$). Thus, due to
$3\delta_4 + 2\delta_3 \le 3n$ (with $|B| = 3n $) and Proposition~\ref{kkos-lemma3} we have

\begin{equation*}
\begin{split}
s(T^{\mathrm{sopt}}) & = m^2 + 3mn + 3m + 15n - 12 + 6\delta_4 + 2\delta_3 \le
m^2 + 3mn + 3m + 21n - 12, \\ 
m (T^{\mathrm{mopt}}) & = m^2 + 7m + 12n - 6 + 6\delta_4 + 2\delta_3 \le
m^2 + 7m + 18n - 6
\end{split}
\end{equation*}
with equality if and only if $\delta_4 = n$ (since  $\delta_3 = 0$ and $6\delta_4 + 2\delta_3 = 6\delta_4 + 4\delta_3 = 6n$), i.e., if and only if the tree
$T^{\mathrm{sopt}}$ (resp. $T^{\mathrm{mopt}}$) arises from a perfect 3-dimensional matching. This yields
the NP-completeness of \textsc{$s$-SF Spanning Tree} and \textsc{$m$-SF Spanning Tree} problems for split graphs.
\end{proof}

It should be noted that Corollary~\ref{kkos-theorem3} implies that both \textsc{$m$-SF Spanning Tree} and \textsc{$s$-SF Spanning Tree}
problems are polynomially solvable for threshold graphs.

Finally, regarding the relations with the max-leaf spanning tree problem, we show that the difference between $\tau_2(G)$ and 
$\max_{T\,\in\, ML(G)} s(T)$ can be arbitrarily large, even within the 
class of split graphs. For an integer $\omega \ge 4$ we construct a split
graph $G_{\omega}$ of order $|G_{\omega}| = 3\omega - 2$ with split
partition $(K, I)$, where $K = \{c_1, c_2, \ldots, c_{\omega}\}$ and 
$I = \{b_1, b_2, \ldots, b_{\omega-1}, b_{\omega}, b_{\omega + 1},
\ldots, b_{2\omega-2}\}$. Each vertex $c_i$, $i = 1, 2, \ldots, \omega - 1$,
is adjacent to the vertices $b_i$ and $b_{i + \omega - 1 }$ and,
additionally, $N_{G_{\omega}}(c_{\omega}) = \{b_{\omega}, b_{\omega + 1},
\ldots, b_{2\omega-2}\}$ (see Fig.~\ref{new-fig7}).
	
\begin{figure}[h]
\centering
\includegraphics[width=6cm]{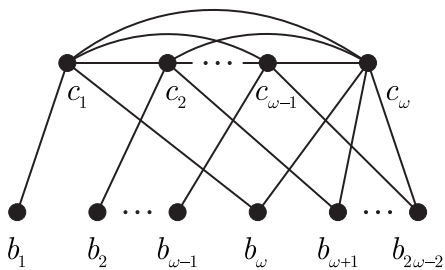}
\caption{The graph $G_{\omega}$}
\label{new-fig7}
\end{figure}

We observe that a minimum connected dominating set of $G_{\omega}$ consists of
vertices $c_1, c_2, \ldots, c_{\omega - 1}$. Firstly each vertex $c_i$,
$i = 1,..., \omega - 1$,  must be included in any minimum connected dominating set
as the only neighbor of $b_i$ (which is not included in minimum connected
dominating set, because of its minimality). And secondly, the set 
$\{c_1, c_2, \ldots, c_{\omega - 1}\}$ of vertices is a connected dominating set of $G_{\omega}$.
This in particular means $c_{\omega}$ is a leaf for any max-leaf spanning tree and
consequently none of the edges $c_{\omega}b_{\omega}, \ldots, c_{\omega}b_{2\omega - 2}$
are in any max-leaf spanning tree. Moreover $\ell(G_\omega) = 2\omega - 1$.
	
Now we produce a new split graph $H_{\omega}$ from $G_{\omega}$ by deleting edges
$c_{\omega}b_{\omega}, \ldots, c_{\omega}b_{2\omega - 2}$. The graph $H_{\omega}$ has
split partition $(K\setminus \{c_{\omega}\}, I \cup \{c_{\omega}\})$. Moreover, every
spanning tree of $G_{\omega}$ with the maximum number of leaves appears to be a spanning
tree of $H_{\omega}$. According to the previous claims, an $s$-optimal tree of $H_{\omega}$
has one of vertices $c_1, c_2, \ldots, c_{\omega - 1}$ as source and vertices 
$b_1, \ldots, b_{2\omega-2}$ and $c_{\omega}$ as leaves. The $s$-optimal tree of
$H_{\omega}$ with source vertex $c_1$ denoted by $T^\mathrm{sopt}_H$ is depicted in
Fig.~\ref{new-fig8} (left). It can be calculated that $s(T^\mathrm{sopt}_H) = 3\omega^2 + 6\omega - 15$ holds.
Since all $s$-optimal trees of $H_{\omega}$ have $2\omega - 1$ leaves and they are clearly
spanning trees of $G_{\omega}$ as well, we have
$\max_{T \in ML(G_\omega)} s(T) = s(T^\mathrm{sopt}_H)$.

On the other hand the $s$-optimal tree $T^\mathrm{sopt}_G$  of $G_\omega$, illustrated in
Fig.~\ref{new-fig8} (right), has source vertex $c_{\omega}$ (due to $\omega \ge 4$, $c_{\omega}$ has
maximum degree in $G_\omega$), $2\omega - 2$ leaves $b_1, \ldots, b_{2\omega-2}$ and $s$-metric
$s(T^\mathrm{sopt}_G) = 6\omega^2 - 10\omega + 4$.

\begin{figure}[h]
\centering
\includegraphics[width=13.5cm]{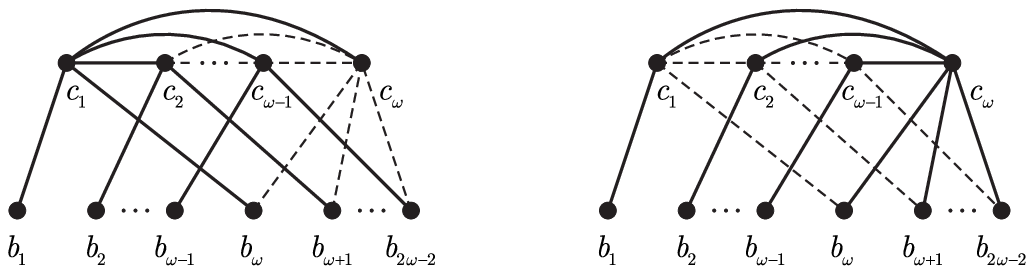}
\caption{The trees $T^\mathrm{sopt}_H$ (left) and $T^\mathrm{sopt}_G$ (right) of $H_{\omega}$ and $G_{\omega}$, respectively}
\label{new-fig8}
\end{figure}

Therefore for each integer $\omega \ge 4$ we have
\begin{equation*}
\tau_2(G_\omega)\,\, - \max_{T \in ML(G_\omega)} s(T) = 3\omega^2 - 16\omega + 19.
\end{equation*}

\section{Integer linear programming formulation and \\heuristics}
\label{scale-free:sec6}

In this section we investigate the practical aspects of scale-free spanning tree problems from the experimental algorithmics perspective. We describe two integer linear programming models and two heuristics for the \textsc{$s$-SF Spanning Tree} problem and conduct computational experiments for various simulated and experimental graphs to evaluate their performance. We concentrate on the \textsc{$s$-SF Spanning Tree} problem, as for the \textsc{$m$-SF Spanning Tree} problem the algorithms are similar. We conclude by demonstrating how the concept of scale-free spanning tree could be used in computational epidemiology for the inference of the history of a viral epidemic spread.
	
For a given spanning tree $T$ of a graph $G=(V,E)$, consider the variables $(x_e)_{e\in E}$ that are defined as follows: 
\begin{equation}\label{eq:charvect}
x_e = \begin{cases}
1, \quad e \in E(T); \\
0, \quad \text{otherwise}.
\end{cases}
\end{equation}

\noindent
Obviously, $T$ contains a path of length 2 or 3 of $G$ if and only if it contains all its edges. This fact and Proposition~\ref{kkos-lemma3} imply that
\begin{equation}\label{eq:s-metr-bool}
s(T) = \sum_{(e_i,e_j,e_k)\, \in\, \Gamma_3(G)} x_{e_i}x_{e_j}x_{e_k} + 2\sum_{(e_i,e_j)\, \in\, \Gamma_2(G)} x_{e_i}x_{e_j} + \sum_{e\, \in\, E(G)} x_e,
\end{equation}
where $\Gamma_i(G)$ denotes the set of all trails of length $i$ in $G$.  In order to linearise \eqref{eq:s-metr-bool} we introduce boolean variables $y_{ijk}$ and $y_{ij}$ and the following constraints:
\begin{equation}\label{constr:lin}
\begin{aligned}
y_{ijk} & \le x_i, & y_{ij} & \le x_i, \\ 
y_{ijk} & \le x_j, & y_{ij} & \le x_j, \\
y_{ijk} & \le x_k, & y_{ij} & \ge x_i + x_j - 1, \\
y_{ijk} & \ge x_i + x_j + x_k - 2, \quad \quad & & 
\end{aligned}
\end{equation}
for every $(e_i,e_j,e_k) \in \Gamma_3(G)$ and $(e_i,e_j) \in \Gamma_2(G)$, which are equivalent to $y_{ijk} = x_{e_i}x_{e_j}x_{e_k}$ and $y_{ij} = x_{e_i}x_{e_j}$. Thus the objective function \eqref{eq:s-metr-bool} can be rewritten as
\begin{equation}\label{odj:lin}
s(T) = \sum_{(e_i,e_j,e_k)\, \in\, \Gamma_3(G)} y_{ijk} + 2\sum_{(e_i,e_j)\, \in\, \Gamma_2(G)} y_{ij} + \sum_{e\, \in\, E(G)} x_e.
\end{equation}

Next, we use two types of constraints to describe the spanning trees.
The first type is  Martin's extended formulation~\cite{Martin91}. Here we use auxiliary variables 
\begin{equation}\label{vars:martin}
z_{(v,w)}^r,\, z_{(w,v)}^r \ge 0 \quad \text{for every}\,\, r \in V(G),\, vw \in E(G),
\end{equation}
where $z_{(v,r)}^r = 0$ for every $r \in V$ and $vr \in E(G)$.
A 0/1-vector $x$ describes a spanning tree of $G$ if and only if there are $z$-variables as in~\eqref{vars:martin} that satisfy the following constraints:
\begin{equation}\label{constr:martin}
\begin{split}
x_{vw} - z_{(v,w)}^r - z_{(w,v)}^r & = 0, \quad r \in V(G),\, vw \in E(G),  \\
\sum_{vw \in E(G)} z_{(v,w)}^r 	& = 1, \quad  r, w \in V(G),\, r \neq w, \\
\sum_{vr \in E(G)} z_{(v,r)}^r	& = 0, \quad r \in V(G).
\end{split}
\end{equation}
	
Another way is to exploit Miller~--~Tucker~--~Zemlin constraints~\cite{MillerTZ60}. We introduce the auxiliary variables 
\begin{equation}\label{vars:miller}
\begin{split}
z_{(v,w)},\, z_{(w,v)} \in \{0,1\} \quad \text{for every}\,\, vw \in E(G),\\
t_v \in [0, n-1] \quad \text{for every}\,\, v \in V(G),
\end{split}
\end{equation}
where $n = |V(G)|$ and constraints
\begin{equation}\label{constr:miller}
\begin{split}
x_{vw} - z_{(v,w)} - z_{(w,v)} & = 0, \quad vw \in E(G), \\
\sum_{vw \in E(G)} z_{(v,w)} 	& = 1, \quad  w \in V(G)\setminus\{r\}, \\
\sum_{vr \in E(G)} z_{(v,r)}	& = 0, \quad r \in V(G), \\
t_v - t_w + nz_{(v,w)} & \le n-1, \quad v,w \in V(G),\, vw \in E(G).
\end{split}
\end{equation}
	
The problem of maximization of the objective \eqref{odj:lin} subject to the constraints \eqref{constr:lin}, \eqref{constr:martin} with auxiliary variables \eqref{vars:martin} will be further referred to as Martin formulation, and the problem with the same objective subject to the constraints \eqref{constr:lin}, \eqref{constr:miller}  with auxiliary variables \eqref{vars:miller} as Miller~--~Tucker~--~Zemlin or MTZ formulation.

We also consider the following two simple greedy heuristics for finding $s$-optimal tree of a graph $G$:
\medskip

\textit{Heuristic-1:} Weight each edge $uv$ of $G$ with $\deg_G u \deg_G v$ and find the maximum-weight spanning tree using Kruskal’s algorithm.

\medskip

\textit{Heuristic-2:} Construct a spanning tree iteratively as follows.  Initialize the algorithm by the tree $T^0$ consisting of all edges incident to the vertex of the maximum degree in $G$. At each next step, choose the vertex $u$ of the previously constructed tree $T^i$ with the maximum number of adjacent vertices outside of $T^i$ and add all edges connecting $u$ to these vertices. The algorithm stops when the current tree spans all vertices of $G$ .
\medskip

Linear programming problems were solved using Gurobi Optimizer Version 8.1. The experiments were conducted using Gurobi Python interface on a standard laptop with 2.0 GHz i7 dual core processor and 16 GB of RAM. Below we describe the results of computational experiments for synthetic and real data-based graphs.

\subsection{Synthetic graphs}

We used graphs from the following synthetic datasets:
\medskip

\textit{Erd\H{o}s~--~R\'{e}nyi graphs}. Those are random $n$-vertex graphs constructed by adding each possible edge uniformly and independently with the probability  $p = 4.25/n$. The number of nodes in our experiments varied from 10 to 40, and the timeout for ILP solver was set to 2400 s.

\medskip

\textit{Grid graphs}. A $n \times m$ grid graph  is a Cartesian product of paths $P_n$ and $P_m$. We explored $4 \times 4$, $4 \times 5$, $5 \times 5$, $5 \times 6$, $6 \times 6$, $6 \times 7$ and $7 \times 7$ grid graphs with timeout of 4500 s.

\medskip

\textit{Scale-free graphs}. We generated scale-free graphs of two types using NetworkX python graph library, which uses the method described in \cite{networkx}. The two explored types were scale-free graphs corresponding to the classical Barab{\'a}si~--~Albert model \cite{barabasi1999emergence} and scale-free graphs with NetworkX default parameters values (after removal of loops and multiple edges), with the latter graphs being denser. The timeout has been set to $1800$ s. 

For all synthetic datasets except for grid graphs we generated 10 graphs per numbers of nodes.

\begin{figure}[h]
\includegraphics[width=.49\textwidth]{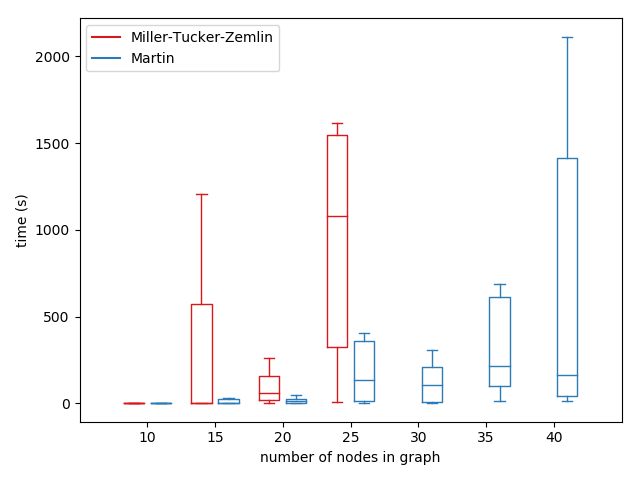} \hfill
\includegraphics[width=.49\textwidth]{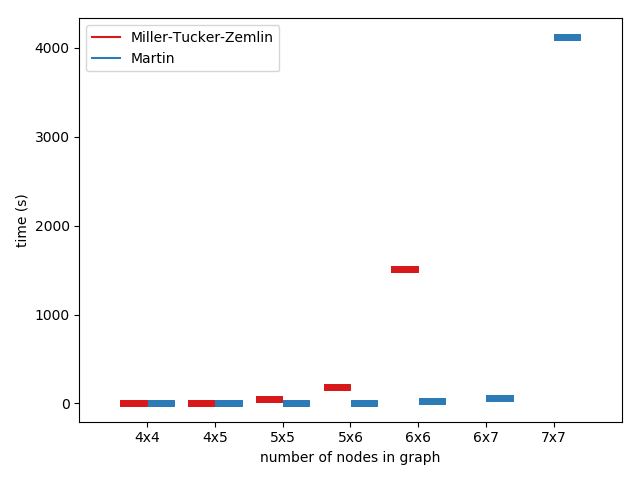} 
\caption{Running times of the ILP solver for two ILP problem formulations. Left to right: Erd\H{o}s~--~R\'{e}nyi graphs and grids}
\label{kkos-figure7-1}
\end{figure}

\begin{figure}[h]
\includegraphics[width=.49\textwidth]{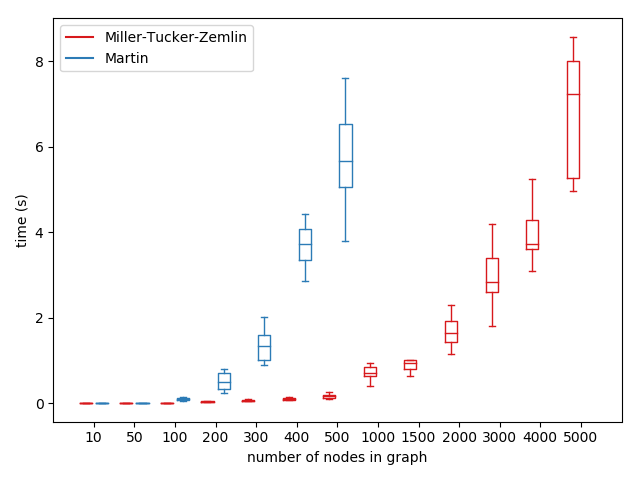} \hfill
\includegraphics[width=.49\textwidth]{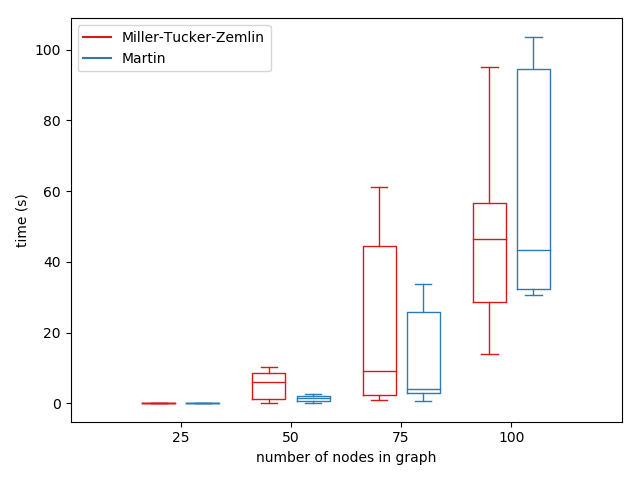}
\caption{Running times of the ILP solver for two ILP problem formulations. Left to right:  Barab{\'a}si~--~Albert scale-free graphs and NetworkX scale-free graphs}
\label{kkos-figure7-2}
\end{figure}

Figures~\ref{kkos-figure7-1},~\ref{kkos-figure7-2}  illustrate the running times of Integer Linear Programming solvers based on MTZ formulation and Martin formulation for all four simulated graph classes.\footnote{Running times for MTZ formulation on grids and Martin formulation on Barab{\'a}si~--~Albert scale-free graphs are plotted only for smaller $n$, since for large values they are significantly higher than for the other formulation. In particular, Martin formulation on Barab{\'a}si~--~Albert scale-free graphs works $\sim$ 150 s for 1000 vertices, $\sim$ 480 s for 1500 vertices and exceeds timeout of 1800 s for 2000 and more vertices.}
The results demonstrate that for those graph models the ILP algorithms in average perform much better than in the worst case and are able to produce optimal results in a reasonable amount of time. For Erd\H{o}s~--~R\'{e}nyi graphs and grids (see Fig.~\ref{kkos-figure7-1}), which are characterized by relatively large sets of feasible solutions, the Miller~--~Tucker~--~Zemlin formulation was superior, while for scale-free graphs (see Fig.~\ref{kkos-figure7-2}) the result of the comparison was the opposite, with Martin's formulation leading to the faster algorithm. In general, ILP allows to solve the problem within minutes or few hours for small-to-medium size problems (up to several dozens of vertices) on Erd\H{o}s~--~R\'{e}nyi graphs and grids, and for medium size problems (several hundred vertices) for scale-free graphs.

Finally, we analyzed the quality of solutions produced by two proposed heuristics on simulated data. For each heuristic solution $T^h$, the approximation ratio $\alpha(T^h)$ was calculated  in comparison to the optimal solutions produced by the exact ILP-based algorithm, i.e.  $\alpha(T^h) = s(T^{sopt}) / s(T^h)$, where $T^{sopt}$ is an optimal solution. The average approximation ratios over the graphs of the same vertex set size are shown on Figures~\ref{kkos-figure8-1},~\ref{kkos-figure8-2}.  For scale-free graphs (see Fig.~\ref{kkos-figure8-2}), both heuristics produce near-optimal solutions for all tested problem sizes.  In contrast, for Erd\H{o}s~--~R\'{e}nyi graphs and grids (see Fig.~\ref{kkos-figure8-1}), the accuracy was lower and significantly declined with the growth of $n$. Thus, these results demonstrate the efficiency of simple heuristic approaches for scale-free graphs and their more limited applicability for Erd\H{o}s~--~R\'{e}nyi and grid graphs.

\begin{figure}[h]
\includegraphics[width=.50\textwidth]{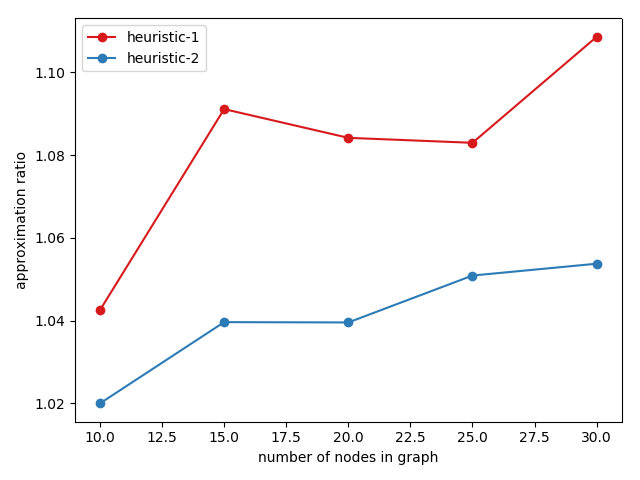}\hfill
\includegraphics[width=.50\textwidth]{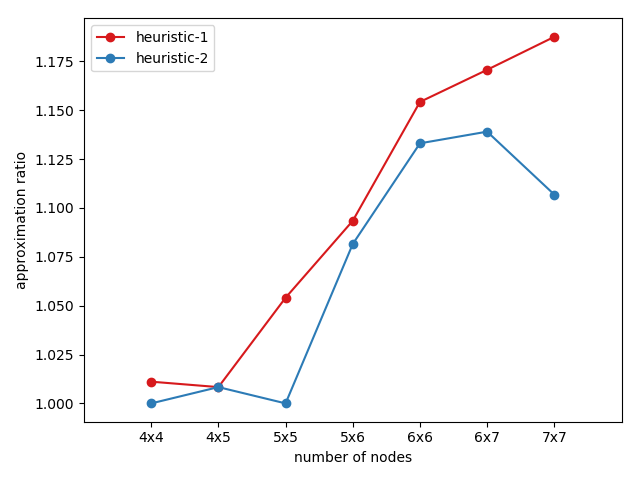} 
\caption{Approximation ratios of two heuristics. Left to right: Erd\H{o}s~--~R\'{e}nyi graphs and grids}
\label{kkos-figure8-1}
\end{figure}

\begin{figure}[h]
\includegraphics[width=.49\textwidth]{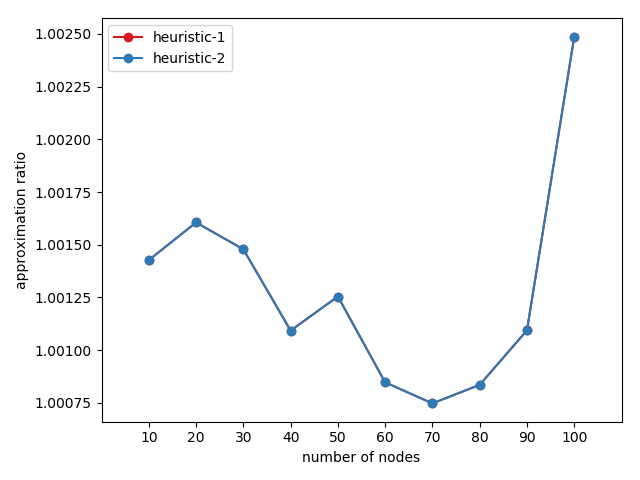} \hfill
\includegraphics[width=.49\textwidth]{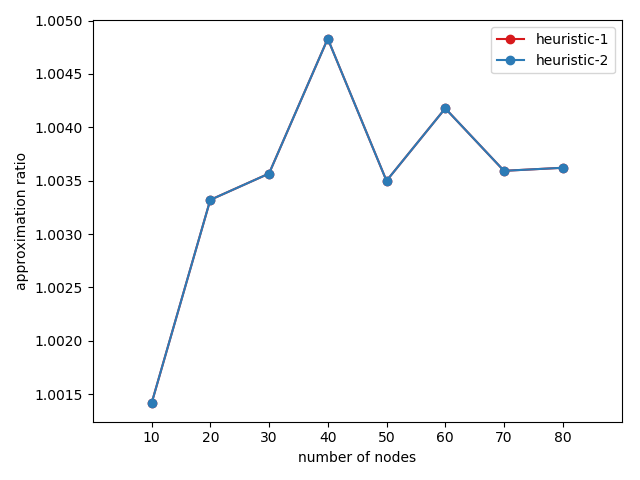}
\caption{Approximation ratios of two heuristics. Left to right: Barab{\'a}si~--~Albert scale-free graphs and NetworkX scale-free graphs}
\label{kkos-figure8-2}
\end{figure}

\subsection{Real data-based graphs}
We applied the concept of scale-free spanning trees to the graphs arising in the area of computational molecular epidemiology.  These graphs correspond to the transmission history reconstruction problem and have been constructed using the dataset consists of RNA sequences of Hepatitis C HVR1 genomic region of length 264 nucleotides sampled from 81 infected individuals involved in seven viral outbreaks \cite{skums2017quentin}. The vertices of each graph correspond to individuals, and two vertices $u$ and $v$ are adjacent, if the minimal relative Hamming distance between the sets of sequences sampled from these patients does not exceed the threshold $t = 3.625\%$. Here we follow the method of graph construction and the threshold value proposed in \cite{campo2015accurate}. In the obtained graph, eight connected components has been identified. Six of these components correspond to the outbreaks, while the seventh outbreak produced two components. For each connected component $C$, its own threshold $t_C$ was defined as the minimal value such that removal of edges $E_C$ corresponding to the distances greater than $t_C$ preserves the connectivity of this component. After removal of edges $E_C$, the ILP algorithm for Martin formulation has been run independently for each connected component. Optimal solutions has been obtained for all analyzed graphs within several hours. For six outbreaks, the superspreaders (the individuals who infected the majority of other individuals) are known from epidemiological investigations \cite{campo2015accurate}. Importantly, those superspreaders correspond to vertices of highest degrees in $s$-optimal trees for five out of six outbreaks. It indicates, that $s$-optimal trees indeed provide epidemiologically accuare and relevant information about transmission histories of viral outbreaks.

\section{Open problems}
\label{scale-free:sec7}

The first open problem is to identify non-trivial graph classes where \textsc{$m$-SF Spanning Tree} and  \textsc{$s$-SF Spanning Tree} problems are polynomially solvable. The analogy with the max-leaf spanning tree problem, for which very few such classes are known, suggests that this may be difficult for the problems under consideration as well. At the same time, the max-leaf spanning tree problem can be approximated within a constant factor thus suggesting the second open problem: verify whether constant or logarithmic approximation exists for \textsc{$m$-SF Spanning Tree} and  \textsc{$s$-SF Spanning Tree} problems. One possible way to investigate this problem is to verify whether $\tau_i(G) \leq c \max_{T \in ML(G)} s(T)$  for some constant $c$. At least it could be claimed that, for instance, the class of graphs where the $s$-optimal tree has the maximum number of leaves is quite rich. Indeed, for any connected graph $H$ there exist infinitely many graphs $G$ for which $\tau_2(G)$ is reached on the spanning tree with the maximum number of leaves and which contain $H$ as an induced subgraph. As an example of such a graph $G$ we can take the corona $H \circ \overline{K}_t$ for some integer $t \ge 1$. Another example of such graph $G$ can be described as follows. Take $n$ disjoint copies (where $n$ is the order of $H$) of a nontrivial tree $T$ with one vertex $r$ chosen as root of $T$ turning $T$ into a rooted tree. Then the graph $G$ can be obtained by identifying the $i$th vertex of $H$ with the root $r$ in the $i$th copy of $T$. It is easy to verify that $G$ has the desired property.



\bibliographystyle{ieeetr}
\bibliography{refs}



\end{document}